\documentclass[11pt]{tearem}

\newcommand\downset{\mathord{\downarrow}}
\newcommand\upset{\mathord{\uparrow}}
\renewcommand{\diamond}{\lozenge}
\newcommand{\pprec}{\mathrel{{\prec}\mkern-5mu{\prec}}}

\let\<\langle
\let\>\rangle

\let\O\relax
\let\P\relax
\let\Omega\Relax
\defOrd[\mathbcal]{\O}{O} 
\defOrd[\mathbcal]{\C}{C} 
\defOrd[\mathbcal]{\Z}{Z} 
\defOrd[\mathbcal]{\at}{at}
\defOp{\Coz}{Coz} 

\defOrd[\mathscr]{\P}{P} 
\defOrd[\mathscr]{\Omega}{O} 
\defOrd[\mathscr]{\CP}{C} 
\defOp[\mathbf]{\RO}{RO}
\let\int\relax
\defOp[\mathsf]{\int}{int}
\defOp[\mathsf]{\cl}{cl}

\defOrd[\mathcal]{\F}{F} 
\defOrd[\mathscr]{\J}{J} 

\defOrd[\mathscr]{\A}{A} 
\defOrd[\mathcal]{\w}{w} 

\defOrd[\mathcal]{\B}{B} 
\defOrd[\mathcal]{\BC}{B} 
\defOrd[\mathcal]{\BO}{B^{\square}} 

\let\max\relax
\defOp[\mathsf]{\max}{max}

\newcommand{\cat}{\mathbf}
\defOp[\cat]{\Top}{Top}
\defOp[\cat]{\SMT}{SMT}
\defOp[\cat]{\MT}{MT}
\defOp[\cat]{\Frm}{Frm}

\title{A choice-free approach to Wallman compactifications}

\author[1]{Sebastian\ D.\ Melzer}
\author[2,3]{Cerene Rathilal}
\author[2]{Ranjitha Raviprakash}

\affil{Department of Mathematics, University of Salerno, Italy}
\affil[2]{University of KwaZulu-Natal, School of Agriculture and Science, Discipline of Mathematics, Private Bag X54001, Durban 4000, South Africa.}
\affil[3]{National Institute for Theoretical and Computational Sciences (NITheCS), South Africa}

\begin{document}

\maketitle

\begin{abstract}
    The classical Wallman compactification of a $T_1$-space and the Stone--\v{C}ech compactification of a completely regular space rely on choice principles. We show that, by representing a space by its powerset MT-algebra (McKinsey--Tarski algebra), both constructions admit choice-free compactifications. More generally, from any Wallman basis of a spatial $T_1$ MT-algebra we construct a compact $T_1$ MT-algebra which is a compactification of the original algebra. Taking the basis of all closed elements yields a choice-free Wallman compactification of every spatial $T_1$ MT-algebra, while taking the basis of zero-elements yields a choice-free Stone--\v{C}ech compactification of every spatial completely regular MT-algebra. Choice is only needed to show that the resulting compactifying algebras are spatial, and hence to recover the usual compactifying spaces. We also show that these constructions recover the corresponding compactifications of frames of opens.
\end{abstract}

\tableofcontents

\section{Introduction}
Compactifications are among the central constructions of general topology. 
They replace a space by a compact extension while retaining the original space as a dense subspace. Classical examples include the Alexandroff one-point compactification, the Wallman compactification of a $T_1$-space, and the Stone--\v{C}ech compactification of a completely regular space (see, e.g., \cite[Sec.~6]{Eng89}). 
Unlike the Alexandroff construction, the Wallman and Stone--\v{C}ech compactifications rely on appropriate choice principles. In particular, the existence of the Wallman compactification for arbitrary $T_1$-spaces cannot be proved without choice \cite{KT13}. In the setting of frames and locales, the corresponding compact objects can be constructed without choice
(see, e.g., \cite{BM80,Joh84,Ban90b}), but the classical compactifying spaces are recovered only after an appropriate choice principle is invoked.

McKinsey--Tarski algebras, or MT-algebras, were introduced in \cite{BR23} as an alternative pointfree approach to topology. A topological space $X$ is represented by the MT-algebra $\P(X)$, its powerset algebra equipped with the interior operator. Spatial MT-algebras correspond to topological spaces, while the open elements of an arbitrary MT-algebra form a frame, and every frame arises in this way. 
Thus, MT-algebras provide a setting in which one may ask whether the classical compactification constructions admit choice-free analogues that retain more information than their frame-theoretic counterparts.

The main result of this paper is the construction of choice-free Wallman compactifications for spatial MT-algebras. From any Wallman basis of a spatial $T_1$ MT-algebra, we construct a compact $T_1$ MT-algebra which is a compactification of the original algebra. Taking the basis of all closed elements yields a choice-free Wallman compactification of every spatial $T_1$ MT-algebra. In particular, if $X$ is a $T_1$-space, then $\P(X)$ admits a choice-free Wallman compactification in the category of MT-algebras. Taking instead the basis of zero-elements yields a choice-free Stone--\v{C}ech compactification of every spatial completely regular MT-algebra, and hence of $\P(X)$ whenever $X$ is completely regular. The construction is carried out using the regular open algebra of the poset of proper filters of the chosen Wallman basis, and is in this respect related to the choice-free Stone duality of Bezhanishvili and Holliday \cite{BH20}.

These compactifications do not, in the absence of choice, necessarily correspond to compactifying spaces. Choice enters when one shows that the resulting compact MT-algebras are spatial. Once spatiality is available, the Wallman compactification of $\P(X)$ corresponds to the classical Wallman compactification of $X$, and the Stone--\v{C}ech compactification of $\P(X)$ corresponds to the classical Stone--\v{C}ech compactification. Thus, the MT-algebraic construction gives choice-free compactifications of the algebras representing spaces even when the corresponding classical compactifying spaces cannot be obtained without choice principles.

We also show that these constructions recover the pointfree compactifications on frames of opens. For a $T_1$-algebra $M$, the frame of opens of its Wallman compactification is isomorphic to the Wallman compactification of $\O(M)$ \cite{Joh84}. Similarly, for a completely regular MT-algebra $M$, the frame of opens of its Stone--\v{C}ech compactification is isomorphic to the Stone--\v{C}ech compactification of $\O(M)$ \cite{BM80}.

To place these constructions in the general theory of MT-algebras, we introduce compactifications of MT-algebras as dense MT-embeddings from compact MT-algebras. For spatial MT-algebras, this recovers the usual notion of compactification of spaces. We show that every MT-algebra admits a compactification by constructing an MT-algebraic analogue of the Alexandroff one-point compactification. On the other hand, assuming the Axiom of Choice, N\"obeling's Spatiality Theorem \cite{Noeb54} implies that non-spatial MT-algebras admit no $T_1$-compactifications. In this sense, the spatiality hypothesis in our Wallman and Stone--\v{C}ech compactification theorems is unavoidable.

The paper is organized as follows. In \cref{sec:prelim}, we recall the necessary facts about frames, MT-algebras, separation axioms, and the functors relating MT-algebras to spaces and frames. In \cref{sec:compactifications}, we introduce compactifications of MT-algebras, construct the Alexandroff extension, and establish the obstruction to $T_1$-compactifications of non-spatial MT-algebras. In \cref{sec:wallman}, we introduce Wallman bases and prove the choice-free Wallman compactification theorem. We then derive the Stone--\v{C}ech analogue, relate both constructions to the corresponding compactifications of frames of opens, and recover the classical compactifications under choice.

\section{Preliminaries}
\label{sec:prelim}

In this section, we recall the necessary background on frames and MT-algebras. We also recall N\"obeling's Spatiality Theorem for MT-algebras, derive Isbell's Spatiality Theorem for frames, and use it to explain the relation between the pointfree and classical Wallman compactifications. For general background on frames, see \cite{PP12,PP21}; for MT-algebras, see \cite{BR23}.

A \emph{frame} is a complete lattice $L$ satisfying the infinite distributive law
\[
a\wedge \bigvee S=\bigvee\{a\wedge s\mid s\in S\}
\]
for all $a\in L$ and $S\subseteq L$. A \emph{frame homomorphism} is a map between frames that preserves finite meets and arbitrary joins. We write $\Frm$ for the category of frames and frame homomorphisms. The motivating example is the frame $\Omega(X)$ of open subsets of a topological space $X$. A frame is \emph{spatial} if it is isomorphic to $\Omega(X)$ for some space $X$.

For an element $a$ of a frame $L$, let
\[
a^*=\bigvee\{b\in L\mid a\wedge b=0\}
\]
be its pseudocomplement.

\begin{definition}[Separation axioms for frames]
A frame $L$ is

\begin{enumerate}
\item \emph{subfit} if $a\nleq b$ implies there exists $c\in L$ such that $a\vee c=1$ and $b\vee c\neq 1$.

\item \emph{regular} if  $a=\bigvee\{b\in L\mid b\prec a\}$,
for each $a \in L$, where $b\prec a$ iff $b^*\vee a=1$.

\item \emph{completely regular} if $a=\bigvee\{b\in L\mid b\pprec a\}$
for each $a \in L$,
where $b \pprec a$ iff there is a family $\{c_p\mid p\in [0,1]\cap\mathbb{Q}\}\subseteq L$ such that $b\leq c_0$, $c_1\leq a$, and $c_p\prec c_q$ whenever $p<q$.
\item \emph{normal} if $a\vee b=1$ implies that there exist $u,v\in L$ such that $u\wedge v=0$, $a\vee u=1$, and $b\vee v=1$.
\end{enumerate}
\end{definition}

An element $a$ of a frame $L$ is \emph{compact} if, whenever $a\leq \bigvee S$ for some $S\subseteq L$, there is a finite $T\subseteq S$ such that $a\leq\bigvee T$. We call $L$ \emph{compact} if its top element is compact. A frame homomorphism $h\colon L\to L'$ is \emph{dense} if $h(a)=0$ implies $a=0$.

\begin{definition}
A \emph{compactification} of a frame $L$ is a pair $(K,h)$, where $K$ is a compact frame and $h\colon K\to L$ is a dense onto frame homomorphism. 

\end{definition}

\begin{remark}
The term \emph{compactification} is often used with additional assumptions on the frame $K$. In this paper, we follow the more general convention (see, e.g., \cite{PP21}) that a compactification is simply a dense onto homomorphism from a compact frame, and we state any further separation requirements explicitly when needed. For example, if $K$ is regular, then we call $(K,h)$ a \emph{regular compactification}.
\end{remark}

We now recall the Wallman compactification of a subfit frame (see, e.g., \cite[pp.~143--147]{PP21}). Let $L$ be a frame. For $u, v \in L$, we say that $u$ is \emph{$v$-small} if, for every $w \in L$, $u \vee w = 1$ implies $v \vee w = 1$. An element $u \in L$ is \emph{suited} if $u = \bigvee\{v \in L \mid v \text{ is $u$-small}\}$. We denote the collection of suited elements of a frame $L$ by $L_s$.
For a frame $L$, denote by $\J(L)$ its ideal frame. Then $\J(L)_s$ is compact and subfit. If $L$ is subfit, then every principal ideal is a suited element, and the join map restricts to a dense onto frame homomorphism $\nu : \J(L)_s \to L$.

\begin{theorem}[Wallman compactification for frames]
\label{thm:frm-wallman}
Let $L$ be a subfit frame. Then $(\J(L)_s,\nu)$ is a subfit compactification of $L$.
\end{theorem}

For a subfit frame, we will call $\w L := \J(L)_s$ the \emph{Wallman extension} of $L$. The construction of $\w L$ does not require choice. Its relation with the classical Wallman compactification of a $T_1$-space will follow below from Isbell's Spatiality Theorem. 

\begin{definition}[McKinsey--Tarski algebras]
\leavevmode
\begin{enumerate}
    \item An \emph{MT-algebra} is a pair $(M,\square)$, where $M$ is a complete boolean algebra and $\square$ is an interior operator on $M$, that is,
\[
\square 1=1,\qquad
\square(a\wedge b)=\square a\wedge \square b,\qquad
\square a\leq a,\qquad
\square a\leq \square\square a
\]
for all $a,b\in M$. 
\item An \emph{MT-morphism} $h\colon M\to N$ is a complete boolean homomorphism between MT-algebras satisfying
$h(\square a)\leq \square h(a)$
for every $a\in M$. 
\end{enumerate}
\end{definition}

As usual, we write simply $M$ for $(M,\square)$ and put $\diamond a :=\neg\square\neg a$. An element $a\in M$ is \emph{open} if $a=\square a$ and \emph{closed} if $a=\diamond a$. We write $\O(M)$ and $\C(M)$ for the collections of open and closed elements of $M$, respectively. The open elements form a frame, and every MT-morphism restricts to a frame homomorphism between the corresponding frames of open elements. Thus, writing $\MT$ for the category of MT-algebras and MT-morphisms, the assignment $M\mapsto \O(M)$ defines a functor from $\MT$ to $\Frm$.

\begin{theorem}[\cite{BR23}]
The functor $\O\colon \MT\to\Frm$ is essentially surjective.
\end{theorem}

The essentially surjectivity is shown as follows: For each frame $L$, the MT-algebra $M$ such that $L \cong \O M$ can be constructed by taking the {\em Funayama envelope} $\F L$ of $L$ \cite{funayama59}. One construction of $\F L$ is to take the MacNeille completion of the boolean envelope of $L$ \cite[Sec.~II.4]{Gra2011}.
In particular, for each frame $L$, taking the MacNeille completion of its boolean envelope, yields an MT-algebra $\F L$ whose frame of open elements is isomorphic to $L$, witnessing essential surjectivity. This is related to Funayama's Theorem \cite{funayama59}, hence we call $\F L$ the \emph{Funayama envelope} of $L$.

For a topological space $X$, let $\P X$ denote the MT-algebra $(\P(X),\int)$. If $f\colon X\to Y$ is continuous, then the inverse image gives an MT-morphism.
This defines a contravariant functor $\P\colon \Top\to\MT$.
Conversely, for an MT-algebra $M$, let $\at(M)$ be the set of atoms of $M$. For $a\in M$, put $\eta_M(a)=\{x\in\at(M)\mid x\leq a\}$. The collection $\eta_M[\O(M)]$ is a topology on $\at(M)$. If $h\colon M\to N$ is an MT-morphism, then the left adjoint of $h$ restricts to a continuous map $\at(h)\colon \at(N)\to \at(M)$. Thus, $\at$ is a contravariant functor from $\MT$ to $\Top$.
We call an MT-algebra \emph{spatial} if its boolean reduct is atomic, and write $\SMT$ for the full subcategory of $\MT$ consisting of spatial MT-algebras. Not every MT-algebra is spatial: for instance, the regular open algebra $\RO(X)$ of a space without isolated points is an atomless complete Boolean algebra (see, e.g., \cite{Joh84,PP12}), and hence non-spatial even when equipped with the discrete interior operator.

\begin{theorem}[\cite{BR23}]
\label{dualadjunctionMT&Top}
The functors $\P$ and $\at$ yield a dual adjunction between $\Top$ and $\MT$. This adjunction restricts to a dual equivalence between $\Top$ and $\SMT$.
\end{theorem}

Let $M$ be an MT-algebra. For $a,b\in M$, we write $a\prec b$ if $\diamond a\leq \square b$, and write $a\pprec b$ if there is a family $\{c_p\mid p\in [0,1]\cap\mathbb{Q}\}\subseteq M$ such that $a\leq c_0$, $c_1\leq b$, and $c_p\prec c_q$ whenever $p<q$. Moreover, if desired, the  family $\{c_p\}$ may be chosen to consist entirely of open elements, or entirely of closed elements. If $a$ and $b$ are open elements, then these notions coincide with the corresponding relations $\prec$ and $\pprec$ on the frame $\O(M)$.

We say that subset $S$ of an MT-algebra $M$ \emph{join-generates} $M$ if every element of $M$ is the join of a family of elements from $S$. We also say that two elements of an MT-algebra are \emph{disjoint} if their meet is $0$.

\needspace{2em}
\begin{definition}[Separation axioms for MT-algebras]
An MT-algebra $M$ is
\begin{enumerate}
\item \emph{$T_1$} or a \emph{$T_1$-algebra} if its closed elements join-generate $M$.
\item \emph{Hausdorff} or a \emph{$T_2$-algebra} if the elements $a=\bigwedge\{\diamond u\mid a\leq u\in\O(M)\}$ join-generate $M$.
\item \emph{regular} if it is $T_1$ and $
u=\bigvee\{v\in\O(M)\mid v\prec u\}$
for each $u\in\O(M)$.
\item \emph{completely regular} if it is $T_1$ and $u=\bigvee\{v\in\O(M)\mid v\pprec u\}$
for each $u\in\O(M)$.
\item \emph{normal} if it is $T_1$ and for all disjoint $c,d\in\C(M)$ there exist disjoint $u,v\in\O(M)$ such that $c\leq u$ and $d \leq v$.
\end{enumerate}
\end{definition}

The following lemma will be used in \cref{sec:compactifications}.

\begin{lemma}
    Let $h: M \to N$ be an onto MT-morphism.
    \begin{enumerate}
    [cref=lemma]
        \item If $M$ is $T_1$, so is $N$.\label{prop:T1 under onto MTmorphism}
        \item If $M$ is Hausdorff, so is $N$. \label{prop:T2 under onto MTmorphism}
    \end{enumerate}
\end{lemma}

\begin{proof}
     Let $0 \neq a \in N$, then there exists $ 0 \neq b \in M$ such that $h(b)=a$. 
     
     \tref{prop:T1 under onto MTmorphism} Now suppose $M$ is $T_1$, then $b=\bigvee\{c \in \C(M) \mid c \leq b\}$, so there exists $c \in \C(M)$ such that $c \leq b$ and $h(c) \neq 0$. Since $h(c) \in \C(N)$ and $h(c) \leq h(b)=a$, every non-zero element of $N$ lies above a closed element, and hence $N$ is $T_1$.

    \tref{prop:T2 under onto MTmorphism} Suppose $M$ is Hausdorff. Then similar to \tref{prop:T1 under onto MTmorphism}, there exists $c \leq b$ such that $c = \bigwedge\{\diamond u \mid c \leq u \in \O(M)\}\}$ and 
    and $h(c) \neq 0$. Since 
    $h(\diamond u) \geq \diamond h(u)$ for any MT-morphism,
    \begin{align*}
        h(c) &= \bigwedge\{h(\diamond u) \mid c \leq u \in \O(M)\} \\
        &\geq \bigwedge\{\diamond h(u) \mid c \leq u \in \O(M)\} \\
        &\geq \bigwedge\{\diamond v \mid h(c) \leq v \in \O(N)\} \\
        &\geq h(c),
    \end{align*}
    so $h(c) = \bigwedge\{\diamond v \mid h(c) \leq v \in \O(N)\}$ 
    and $h(c) \leq a$. Thus, every non-zero element of $N$ lies above an element of this form, and hence $N$ is Hausdorff.
\end{proof}

The separation axioms of MT-algebras and frames are closely related. In particular, a $T_1$-algebra is regular, completely regular, or normal iff its frame of opens has the corresponding property. Moreover, the Funayama envelope of a frame is $T_1$ iff the frame is subfit. This yields the following.

\needspace{2em}
\begin{proposition}
    Let $L$ be a frame.
    \begin{enumerate}[cref=proposition]
        \item $L$ is subfit iff $\F L$ is $T_1$.\label{prop:subfit-iff-t1}
        \item $L$ is regular iff $\F L$ is regular.
        \item $L$ is completely regular  iff $\F L$ is completely regular.
    \end{enumerate}
\end{proposition}

\begin{remark}
Under the assumption of subfitness, analogous statements hold for Hausdorffness (in the sense of \cite[p.~43]{PP21}) and normality. 
\end{remark}

The following will be used in \cref{sec:wallman} and is a direct consequence of \cite[Thm.~6.5]{BR23}.

\begin{theorem} \label{lem:T1-iso}
    If $M$ is a $T_1$-algebra, then $M \cong \F\O(M)$. In particular, for $T_1$-algebras $M$ and $N$, the following are equivalent.
    \begin{enumerate}[cref=lemma]
        \item $M \cong N$;
        \item $\O(M)\cong \O(N)$;
        \item $\C(M) \cong \C(N)$.
    \end{enumerate}
\end{theorem}

An element $a$ of an MT-algebra $M$ is \emph{compact} if $a\leq \bigvee S$ for some $S\subseteq \O(M)$, implies there is a finite $T\subseteq S$ such that $a\leq\bigvee T$. We call $L$ \emph{compact} if its top element is compact. 
It is clear that an MT-algebra is compact iff its frame of open elements is compact. 
It is straightforward to verify that following well-known facts about compact subsets extend to the setting of MT-algebras. 
\begin{lemma}[{see, e.g., \cite[Lems.~3.6 and 6.11]{BR+25}}]
    Let $M$ be an MT-algebra.
    \begin{enumerate}[cref=lemma]
        \item If $c \in \C(M)$ and there exists compact $k$ with $c \leq k$, then $c$ is compact.\label{closed-below-compact}
        \item An element $a \in M$ is compact iff for every $S \subseteq \C(M)$ such that $a \wedge \bigwedge S = 0$ there is a finite $T \subseteq S$ such that $a \wedge \bigwedge T = 0$.\label{compactness-meets}
    \end{enumerate}
\end{lemma}

The following theorem \cite[12.5]{Noeb54} is the spatiality result needed below. For the convenience of the reader, we prefix results that are choice-dependent by an asterisk.
\begin{*theorem}[N\"obeling's Spatiality Theorem]\label{noeb-spatiality}
    Compact $T_1$-algebras are spatial.
\end{*theorem}
 Isbell's Spatiality Theorem \cite{Isb72} is a direct corollary as the frame of opens of a spatial MT-algebra is spatial (see \cite[Prop.~4.11]{BR23}). In fact, if $M$ is spatial, then $\O(M) \cong \O(\P(\at(M)) = \Omega(\at(M))$, so $\O(M)$ is exactly the frame of opens of the space of atoms of $M$.

\begin{*corollary}[Isbell's Spatiality Theorem] \label{cor:isbell}
    Compact subfit frames are spatial. In particular, for every compact subfit frame $L$ there exists a compact $T_1$-space $X$ such that $L \cong \Omega(X)$.
\end{*corollary}

\begin{proof}
    Let $L$ be a compact subfit frame. Then $\F L$ is a compact $T_1$-algebra by \cref{prop:subfit-iff-t1}, so it is spatial by N\"obeling's Spatiality Theorem. Let $X = \at (\F L)$. Then $L \cong \Omega(X)$, and since $\P(X) \cong \F L$ is compact and $T_1$, so is $X$.
\end{proof}

We can now explain the relation between the Wallman extension of frames and the Wallman extension of spaces. 

\begin{*remark}\label{rem:wallman}
Let $X$ be a $T_1$-space. Recall that the \emph{Wallman extension} of $X$ is the space $\w X$ of maximal filters of $\C(X)$, equipped with the topology having $\{\widehat C \mid C \in \C(X)\}$ as a closed base, where $\widehat C=\{F \in \w X \mid C \in F\}$. 
Since $X$ is $T_1$, the frame $L = \Omega(X)$ is subfit (see, e.g., \cite[p.~73]{PP12}), so by \cref{thm:frm-wallman}, $(\w L,\nu)$ is a subfit compactification of $L$. By \cref{cor:isbell}, $\w L \cong \Omega(Y)$ for some compact $T_1$-space $Y$. The points of $Y$ correspond to the maximal proper elements of $\w L$, which are precisely the maximal ideals of $L$, and hence correspond to the points of $\w X$. Moreover, if $C=X\setminus U$ for $U\in L$, then the open of $Y$ corresponding to the principal suited ideal $\downset U$ is the complement of $\widehat C$. Since the principal suited ideals generate $\w L$, the topology on $Y$ is exactly the Wallman topology. Thus, $Y$ is homeomorphic to $\w X$.
\end{*remark}

\section{Compactifications of MT-algebras}
\label{sec:compactifications}

In this section, we introduce compactifications of MT-algebras and relate them to compactifications of topological spaces and frames. We then construct an MT-algebraic analogue of the Alexandroff one-point compactification. Consequently, we obtain that every MT-algebra has a compactification. 
However, assuming choice, we then observe that the Alexandroff one-point compactification of an MT-algebra can only be $T_1$ if the original algebra is $T_1$ and spatial. This turns out not to simply be an obstruction of the one-point construction, but rather an obstruction of compactifications of MT-algebras in general. Assuming choice, non-spatial MT-algebras do not admit compactifications that are $T_1$. Finally, without assuming choice, we show that the Alexandroff extension of an MT-algebra is a Hausdorff compactification iff the original algebra is non-compact, locally compact, and Hausdorff.

\subsection{Compactifications and their relation to spaces and frames}

We begin by motivating the definition. Recall that a one-to-one continuous map is an \emph{embedding} if it is a homeomorphism onto its image. By a \emph{compactification} of a topological space $X$ we mean a pair $(Y,f)$ such that $Y$ is compact, $f:X \to Y$ is an embedding, and $f(X)$ is dense in $Y$. We call $(Y,f)$ a \emph{$T_1$-compactification} (resp.~\emph{Hausdorff compactification}) if, in addition, $Y$ is a $T_1$-space (resp.~Hausdorff space). The following simple observations describe density and embeddings using inverse-image maps and motivate the corresponding definitions for MT-algebras.
\begin{lemma} \label{lem: cts maps density}
Let $X$ and $Y$ be topological spaces and let $f: X \to Y$ be continuous. Then:
    \begin{enumerate}[cref=lemma]
    \item \label{f dense iff f-1 dense}
    $f(X)$ is dense in $Y$ iff $f^{-1}(\int A)=\varnothing$ implies $\int A=\varnothing$ for every $A\subseteq Y$;
    \item \label{f 1-1 iff f-1 onto}
    $f$ is one-to-one iff $f^{-1}:\P Y\to \P X$ is onto;
    \item \label{lem: embedding}
    $f$ is an embedding iff $f^{-1}:\P Y\to \P X$ is onto and, for every open $U\subseteq X$, there exists an open $V\subseteq Y$ such that $f^{-1}(V)=U$.
    \end{enumerate}
\end{lemma}
\begin{proof}
    \tref{f dense iff f-1 dense} and \tref{f 1-1 iff f-1 onto} are immediate. For \tref{lem: embedding}, by \tref{f 1-1 iff f-1 onto}, it remains only to observe that a one-to-one continuous map $f: X\to Y$ is an embedding iff every open subset of $X$ is the inverse image of an open subset of $Y$, which is straightforward from the definition.
\end{proof}

This leads to the following definition.

\begin{definition}[Compactifications of MT-algebras]
Let $M,N$ be MT-algebras and let $h: N \to M$ be an MT-morphism.
\begin{enumerate}[cref=definition]
\item We call $h$ \emph{dense} if $h(u)=0$ implies $u=0$ for every $u\in \O(N)$.
\item We call $h$ an \emph{MT-embedding} if both $h$ and $\O(h)$ are onto.
\item A \emph{compactification} of $M$ is a pair $(N,h)$, where $N$ is a compact MT-algebra and $h: N \to M$ is a dense MT-embedding. \label{def:compactification}
\end{enumerate}
\end{definition}
\begin{remark}
\leavevmode
\begin{enumerate}
\item
Embeddings are precisely the extremal monomorphisms in the category of topological spaces. Dually, an MT-morphism $h: N \to M$ is an extremal epimorphism iff both $h$ and $\O(h)$ are onto, that is, iff it is an MT-embedding. We will establish this characterization elsewhere. Thus, a compactification of $M$ may equivalently be described as a dense extremal epimorphism $h:N\to M$ from a compact MT-algebra~$N$.

\item
Extremal epimorphisms, and hence compactifications, admit a concrete description in terms of \emph{relativizations} (see, e.g., \cite[pp.~95--97]{RS63}). Let $N$ be an MT-algebra and let $a\in N$. The relativization of $N$ to $a$ is the MT-algebra
\[
    N_a=\downset a:=\{b\in N\mid b\leq a\},
\]
whose frame of open elements is
\[
    \O(N_a)=\{u\wedge a\mid u\in \O(N)\}.
\]
The map $\pi_a:N\to N_a$ given by $\pi_a(b)=b\wedge a$ is an extremal epimorphism, and every extremal epimorphism with domain $N$ is, up to isomorphism of its codomain, of this form. Consequently, $(N,h)$ is a compactification of $M$ iff $N$ is compact and there exist $a\in N$ and an isomorphism $e:N_a\to M$ such that $h=e\circ\pi_a$ and $\pi_a$ is dense.

\end{enumerate}
\end{remark}
By \cref{lem: cts maps density} and \cref{dualadjunctionMT&Top}, compactifications of MT-algebras extend compactifications of topological spaces in the following sense.

\needspace{2em}
\begin{proposition} \label{thm: compactification iffs}
    Let $f : Y \to X$ be a continuous map, let $M,N$ be spatial MT-algebras, and let $h : N \to M$ be an MT-morphism. Then:
    \begin{enumerate}[cref=proposition]
        \item \label{prop:comp-iff-spaces}
        $(Y,f)$ is a compactification of $X$ iff $(\P Y,f^{-1})$ is a compactification of $\P X$;
        \item \label{prop:comp-iff-alg}
        $(N,h)$ is a compactification of $M$ iff $(\at(N),\at(h))$ is a compactification of $\at(M)$.
    \end{enumerate}
\end{proposition}

\begin{proof}
    \tref{prop:comp-iff-spaces} follows from \cref{lem: cts maps density}, together with the fact that $Y$ is compact iff $\P Y$ is compact. For \tref{prop:comp-iff-alg}, since $M$ and $N$ are spatial, $M \cong \P(\at(M))$ and ${N \cong \P(\at(N))}$, and under these isomorphisms $h$ corresponds to $\at(h)^{-1}$. The result therefore follows from~\tref{prop:comp-iff-spaces}.
\end{proof}

We point out that \cref{prop:comp-iff-alg} fails in both directions for non-spatial MT-algebras.

\begin{example} \label{rem: compactifications preserved}
\leavevmode
\begin{enumerate}
    \item Let $M=\P(\mathbb N)\times B$, where $B$ is a complete atomless boolean algebra, and
    \[
        \O(M)=\{(a,0)\mid a\in \P(\mathbb N)\}\cup\{(1,1)\}.
    \]
    Then $M$ is compact since $(1,1)$ is completely join-irreducible in $\O(M)$, but $\at(M)$ is homeomorphic to the space of natural numbers $\mathbb N$ with the discrete topology, so it is not compact. Thus, $M$ paired with the identity morphism is a compactification of $M$, but its $\at(M)$ paired with the identity map is not a compactification of $\at(M)$.

    \item Let $B$ be a non-compact complete atomless boolean algebra and let $M$ be the discrete MT-algebra whose boolean reduct is $B$. Then $M$ paired with the identity morphism is not a compactification, since $M$ is not compact. However, $\at(M)=\varnothing$, and hence $\at(M)$ paired with the identity map is a compactification of $\at(M)$.
\end{enumerate}
\end{example}

\begin{remark}
    By \cref{dualadjunctionMT&Top}, the functors $\P$ and $\at$ are part of a contravariant adjunction. One can check that under this adjunction, extremal epimorphisms in \MT{} are mapped to extremal monomorphisms in \Top{} and vice versa. The failure of \cref{prop:comp-iff-alg} for non-spatial MT-algebras is therefore purely a consequence of the fact that the functor $\at$ neither preserves nor reflects compactness.
\end{remark}

We next compare compactifications of MT-algebras to compactifications of frames. First, we show that compactifications of MT-algebras are preserved under the functor $\O$. In fact, we have the following slightly stronger result:

\begin{proposition} \label{prop: O iff compactifications}
    Let $h : N \to M$ be an onto MT-morphism. 
    Then $(N,h)$ is a compactification of $M$ iff $(\O(N), \O(h))$ is a compactification of $\O(M)$.
\end{proposition}

\begin{proof}
    ($\Rightarrow$) Suppose $(N,h)$ is a compactification. Then $h$ is an MT-embedding, so $\O(h)$ is an onto frame homomorphism. Moreover, $\O(N)$ is compact since $N$ is. Also, if ${\O(h)(u)=0}$ for some $u \in \O(N)$, then $h(u)=0$, so $u=0$ by density of $h$. Thus, $\O(h)$ is dense, and hence $(\O(N),\O(h))$ is a compactification of $\O(M)$.
    
    ($\Leftarrow$) Suppose $(\O(N),\O(h))$ is a compactification of $\O(M)$. Since $h$ is onto by hypothesis and $\O(h)$ is onto, $h$ is an MT-embedding. Since $\O(N)$ is compact, $N$ is compact. Finally, if $u \in \O(N)$ and $h(u)=0$, then $\O(h)(u)=0$, so $u=0$ by density of $\O(h)$. Thus, $h$ is dense, and hence $(N,h)$ is a compactification of $M$.
\end{proof}

The assumption that $h$ is onto in \cref{prop: O iff compactifications} is necessary. Indeed, the right-to-left implication can fail if $h$ is not onto, as shown in the following simple example.

\begin{example}
    Let $N$ be the two-element MT-algebra and let $M$ be the four-element indiscrete MT-algebra, meaning that $\O(M)$ consists only of the top and bottom elements. Then the obvious map $h : N \to M$ is an MT-morphism such that $\O(h)$ is an isomorphism. Hence, $(\O(N),\O(h))$ is a compactification of $\O(M)$. However, $(N,h)$ is not a compactification of $M$, since $h$ is not onto.
\end{example}

When $M$ is a $T_1$-algebra, the closed elements join-generate $M$. Consequently, surjectivity of $\O(h)$ is enough to obtain surjectivity of $h$, and the additional assumption in \cref{prop: O iff compactifications} is no longer needed.

\begin{theorem} \label{thm:O-compactifications-T1}
    Let $M$ be a $T_1$-algebra and let $h : N \to M$ be an MT-morphism. Then $(N,h)$ is a compactification of $M$ iff $(\O(N),\O(h))$ is a compactification of $\O(M)$.
\end{theorem}

\begin{proof}
    ($\Rightarrow$) Suppose $(N,h)$ is a compactification of $M$. Then $h$ is onto, so $(\O(N),\O(h))$ is a compactification of $\O(M)$ by \cref{prop: O iff compactifications}.

    ($\Leftarrow$) Suppose $(\O(N),\O(h))$ is a compactification of $\O(M)$. By \cref{prop: O iff compactifications}, it suffices to show that $h$ is onto. Let $a \in M$. Since $M$ is $T_1$,
    \[
        a = \bigvee\{c \in \C(M) \mid c \leq a\}.
    \]
    Since $\O(h)$ is onto, for each $c \in \C(M)$ with $c \leq a$, there exists $u_c \in \O(N)$ such that $h(u_c)=\neg c$. Since $h$ is a Boolean homomorphism, $h(\neg u_c)=c$. Therefore,
    \[
        h\left(\bigvee\{\neg u_c \mid c \in \C(M),\ c \leq a\}\right)
        = \bigvee\{c \in \C(M) \mid c \leq a\}
        = a.
    \]
    Thus, $h$ is onto.
\end{proof}

\begin{remark}
    The functor $\O$ is not full (see \cite[Lem.~4.6]{BR23}). That is, for two MT-algebras $M,N$ there can be frame homomorphisms $g:\O(N)\to\O(M)$ that do not extend to an MT-morphism $N\to M$. In particular, $(\O(N),g)$ can be a compactification of $\O(M)$ while $g$ is not induced by any MT-morphism $N\to M$. Moreover, this obstruction remains even when restricted to $T_1$-algebras, as we will see in \cref{ex:frame-compactification-does-not-lift}. Thus, compactifications of frames cannot, in general, be described by fixing MT-algebras over their domain and codomain and asking for an inducing MT-morphism. It remains open whether every frame homomorphism $g:K\to L$ can be realized, up to isomorphism, as $\O(h)$ for some MT-morphism $h:N\to M$ with $\O(N)\cong K$ and $\O(M)\cong L$. In particular, it remains open whether every compactification of frames can be realized in this way by a compactification of MT-algebras.
\end{remark}

\subsection{The Alexandroff compactification of an MT-algebra}

In this subsection, we show that every non-compact MT-algebra admits a compactification by constructing an MT-algebraic analogue of the Alexandroff one-point compactification. By $2$ we denote the two-element MT-algebra.

\begin{definition}[Alexandroff extension of an MT-algebra]
Let $M$ be an MT-algebra. Set $\A(M)=M\times 2$ and define an operator $\square$ on $\A(M)$ by
\[
    \square(a,b)=
    \begin{cases}
        (\square a,b) & \text{if } \neg\square a \text{ is compact},\\
        (\square a,0) & \text{otherwise}.
    \end{cases}
\]
We call $\A(M)$ the \emph{Alexandroff extension} of $M$.
\end{definition}

\begin{lemma}
    The operator $\square$ defined above is an interior operator on $\A(M)$. Hence, $\A(M)$ is an MT-algebra.
\end{lemma}

\begin{proof}
    It is clear that $\square(1,1)=(1,1)$ and that $\square(a,b)\leq (a,b)$ for every $(a,b)\in \A(M)$. Moreover, since $\square\square a=\square a$, we have $\square\square(a,b)=\square(a,b)$.
    It remains to show that $\square$ preserves binary meets. Let $(a,b),(c,d)\in \A(M)$. Since
\[
    \neg\square(a\wedge c)=\neg\square a\vee\neg\square c,
\]
and both $\neg\square a$ and $\neg\square c$ are closed. It follows from \cref{closed-below-compact} that $\neg\square(a\wedge c)$ is compact only if both $\neg\square a$ and $\neg\square c$ are compact. The converse follows since finite joins of compact elements are compact.
Therefore,
    \[
        \square((a,b)\wedge(c,d))
        =\square(a,b)\wedge\square(c,d).
    \]
    Thus, $\square$ is an interior operator on $\A(M)$.
\end{proof}

We next show that the Alexandroff extension of a non-compact MT-algebra is a compactification. For this we require the following lemmas.

\begin{lemma}
    Let $M$ be an MT-algebra and $(a,b) \in \A(M)$. 
    \begin{enumerate}[cref=lemma]
        \item $(a,b) \in \O(\A(M))$ iff $a \in \O(M)$ and $b = 1$ implies that $\neg a$ is compact. \label{lem:alex-open}
        \item $(a,b) \in \C(\A(M))$ iff $a \in \C(M)$ and $b = 0$ implies that $a$ is compact. \label{lem:alex-closed} 
    \end{enumerate}
\end{lemma}

\begin{proof}
    \tref{lem:alex-open} By definition, $(a,b)$ is open iff $a=\square a$ and either $b=0$ or $\neg\square a$ is compact. Since $a=\square a$, this is equivalent to $a \in \O(M)$ and $b=1$ implies that $\neg a$ is compact.

    \tref{lem:alex-closed} We have
    \[
        \diamond(a,b)
        =\neg\square(\neg a,\neg b)
        =
        \begin{cases}
            (\diamond a,b) & \text{if } \diamond a \text{ is compact},\\
            (\diamond a,1) & \text{otherwise.}
        \end{cases}
    \]
    Therefore, $(a,b)$ is closed iff $a \in \C(M)$ and $b=0$ implies that $a$ is compact.
\end{proof}

\begin{lemma}\label{alex-compact}
    $\A(M)$ is compact.
\end{lemma}
\begin{proof}
    Suppose $\bigvee (u_i,v_i)=(1,1)$ for some $(u_i,v_i)\in \O(\A(M))$. Then $\bigvee u_i=1$ and ${\bigvee v_i=1}$. The latter implies that $v_i=1$ for some $i$. Therefore, $\neg u_i$ is compact by \cref{lem:alex-open}. Since $\neg u_i \leq \bigvee u_i$, there exist $(u_1,v_1),\dots,(u_n,v_n)$ among the given open elements such that $\neg u_i \leq u_1\vee\cdots\vee u_n$.
    Hence, $(u_i,v_i)\vee (u_1,v_1)\vee\cdots\vee(u_n,v_n)=(1,1)$,
    so $\A(M)$ is compact.
\end{proof}

Let $\pi : \A(M) \to M$ be the projection defined by $\pi(a,b)=a$ for $(a,b) \in \A(M)$. 

\begin{lemma}\label{piextrimalepi}
    $\pi: \A(M) \to M$ is an MT-embedding.
\end{lemma}

\begin{proof}
    It is clearly an onto complete boolean homomorphism. Moreover,
    \[
        \pi(\square(a,b))=\square a=\square\pi(a,b),
    \]
    so it is an MT-morphism. By \cref{lem:alex-open},  $u \in \O(M)$ implies that ${(u,0) \in \O(\A(M))}$. Thus, $\O(\pi)$ is onto, and hence $\pi$ is an MT-embedding.
\end{proof}

We are now ready to show that the Alexandroff extension yields a compactification for each non-compact MT-algebra.

\begin{theorem} \label{Alexandroff compactification}
    If $M$ is not compact, then $(\A(M),\pi)$ is a compactification of $M$.
\end{theorem}

\begin{proof}
    By \cref{piextrimalepi}, $\pi$ is an MT-embedding. By \cref{alex-compact}, $\A(M)$ is compact. To see that $\pi$ is dense, let $(u,v) \in \O(\A(M))$ and suppose that $\pi(u,v)=0$. Then $u=0$. Since $M$ is not compact, $\neg u=1$ is not compact, and hence \cref{lem:alex-open} implies that $v=0$. Thus, $(u,v)=(0,0)$, and so $\pi$ is dense.
\end{proof}

\begin{corollary}
    Every MT-algebra admits a compactification.
\end{corollary}

\begin{proof}
    If $M$ is compact, then $M$ paired with the identity morphism is a compactification of $M$. If $M$ is not compact, then \cref{Alexandroff compactification} applies.
\end{proof}

Using \cref{prop: O iff compactifications} and \cref{Alexandroff compactification}, the Alexandroff extension of MT-algebras yields compactifications of frames. We thus obtain a one-point compactification, in the sense of \cite{Ban90b}, for each non-compact frame.

\begin{corollary}
    Let $L$ be a non-compact frame. Then there exists a compactification $(K,h)$ of $L$ such that $K$ has a maximal proper element $a$ and the restriction $h:\downset a\to L$
    is a frame isomorphism.
\end{corollary}

\begin{proof}
    Let $L$ be a non-compact frame and let $M=\F L$ be its Funayama envelope. Identifying $\O(M)$ with $L$, we have that $M$ is not compact. Hence, $(\A(M),\pi)$ is a compactification of $M$ by \cref{Alexandroff compactification}. Consequently,
    $(\O(\A(M)),\O(\pi))$
    is a compactification of $L$ by \cref{prop: O iff compactifications}. Put $K=\O(\A(M))$ and $h=\O(\pi)$.

    The element $(1,0)$ is a maximal proper element of $K$. Moreover, by \cref{lem:alex-open},
    $\downset(1,0)=\{(u,0)\mid u\in \O(M)\}$,
    and the restriction of $h$ to this downset is given by
    $h(u,0)=u$.
    Thus, $h:\downset(1,0)\to L$ is a frame isomorphism.
\end{proof}

\begin{remark}
    The preceding corollary is an analogue of Banaschewski's frame-theoretic treatment of the Alexandroff one-point compactification for our more general notion of compactification. In \cite[pp.~113--115]{Ban90b}, Banaschewski proves that the regular continuous frames are, up to isomorphism, exactly the frames of the form $\downset a$, where $a$ is a maximal element of a compact regular frame. Our construction gives a compact frame $K$ of this form for every non-compact frame $L$, but does not in general produce a compact regular frame.
\end{remark}

The Alexandroff one-point extension of a $T_1$-space is a compact $T_1$-space. This does not carry over to MT-algebras. Indeed, even the Alexandroff extension of a discrete MT-algebra need not be $T_1$.

\begin{example}[A discrete algebra whose Alexandroff extension is not $T_1$]
    Let $M$ be a discrete MT-algebra whose boolean reduct is atomless. Then the only compact element of $M$ is $0$, so by \cref{lem:alex-closed}, the only $(a,b) \in \C(\A(M))$ with $b=0$ is $(0,0)$. Therefore, there is no non-zero closed element below $(1,0)$. Hence, the closed elements do not join-generate $\A(M)$, and so $\A(M)$ is not $T_1$.
\end{example}

The situation is quite restrictive. Assuming choice, we now show that $\A(M)$ is $T_1$ iff $M$ is $T_1$ and spatial. In this case, $\A(M)$ corresponds to the classical Alexandroff one-point extension of the associated $T_1$-space. For this we first require the following lemma.

\begin{lemma} \label{lem: onto spatial}
    If $h:M\to N$ is an onto MT-morphism and $M$ is spatial, then $N$ is spatial.
\end{lemma}

\begin{proof}
    Let $n\in N$ be non-zero. Since $h$ is onto, there exists $m\in M$ such that $h(m)=n$. Since $M$ is spatial,
    \[
        m=\bigvee \eta_M(m) = \bigvee\{x\in\at(M)\mid x\leq m\}.
    \]
    Therefore,   $n=h(m)=\bigvee\ h[\eta_M(m)]$,
    so there exists an atom $x\in M$ such that $x\leq m$ and $h(x)\neq 0$. We claim that $h(x)$ is an atom of $N$. Let $0\neq y\leq h(x)$. Since $h$ is onto, there exists $b\in M$ such that $h(b)=y$. Then
    \[
        0\neq h(b)=h(b)\wedge h(x)=h(b\wedge x),
    \]
    so $b\wedge x\neq 0$. Since $x$ is an atom, $b\wedge x=x$, and hence
    \[
        y=h(b)=h(b\wedge x)=h(x).
    \]
    Thus, $h(x)$ is an atom of $N$. Since $h(x)\leq h(m)=n$, every non-zero element of $N$ lies above an atom, and hence $N$ is spatial.
\end{proof}
\begin{remark}
    \Cref{lem: onto spatial} is the well-known fact that complete homomorphic images of complete atomic boolean algebras are atomic, applied to the boolean reducts of MT-algebras.
\end{remark}

\begin{*theorem}\label{separationmakesAlexandroffspatial}
Let $M$ be an MT-algebra. The following are equivalent:
\begin{enumerate}[cref=theorem]
    \item $\A(M)$ is $T_1$; \label{alex-spatial-1}
    \item $\A(M)$ is $T_1$ and spatial; \label{alex-spatial-2}
    \item $M$ is $T_1$ and spatial. \label{alex-spatial-3}
\end{enumerate}
\end{*theorem}

\begin{proof}
    \tref{alex-spatial-1}$\Rightarrow$\tref{alex-spatial-2} Suppose $\A(M)$ is $T_1$. Since $\A(M)$ is compact by \cref{alex-compact}, N\"obeling's Spatiality Theorem applies, and hence $\A(M)$ is spatial.

    \tref{alex-spatial-2}$\Rightarrow$\tref{alex-spatial-3} Suppose $\A(M)$ is $T_1$ and spatial. Since $\pi:\A(M)\to M$ is onto by \cref{piextrimalepi}, it follows from \cref{lem: onto spatial} that $M$ is spatial. Since $\pi$ is a surjective MT-morphism, $M$ is $T_1$ by \cref{prop:T1 under onto MTmorphism}. 
    
    \tref{alex-spatial-3}$\Rightarrow$\tref{alex-spatial-1} Suppose $M$ is $T_1$ and spatial, and let $(a,b)\in\A(M)$ be non-zero. If $b=1$, then $(0,1)\leq(a,b)$, 
    and $(0,1)\in\C(\A(M))$ by \cref{lem:alex-closed}. Suppose $b=0$. Then $a\neq 0$, so there exists an atom $x\in M$ such that $x\leq a$. Since $M$ is $T_1$, there exists a non-zero closed element below $x$, and hence $x$ itself is closed. Moreover, $x$ is compact because it is an atom. Thus, $(x,0)\in\C(\A(M))$ by \cref{lem:alex-closed}, and
    $(x,0)\leq(a,b)$.
    Therefore, every non-zero element of $\A(M)$ lies above a non-zero closed element, and hence $\A(M)$ is $T_1$.
\end{proof}

This shows that it is quite difficult for $\A(M)$ to satisfy separation axioms. As soon as we require $T_1$, we are limited to the spatial case. We return to the Hausdorff case shortly. In fact, as the reader might have observed, there was nothing specific to $\A(M)$ in the part of the proof of \cref{separationmakesAlexandroffspatial} showing that $M$ is spatial. We thus make the following observation.

\begin{*theorem} \label{cor:no-compactification}
    A non-spatial MT-algebra does not admit a $T_1$-compactification.
\end{*theorem}

\begin{proof}
    Suppose $(N,h)$ is a $T_1$-compactification of $M$. Then $N$ is compact and $T_1$, hence spatial by N\"obeling's Spatiality Theorem. Since $h$ is onto, \cref{lem: onto spatial} implies that $M$ is spatial.
\end{proof}

\begin{*example}\label{ex:frame-compactification-does-not-lift}
    Let $L$ be a non-spatial subfit frame, for example an atomless complete boolean algebra, and let $(\w L,\nu)$ be its Wallman compactification. Then $\F L$ is a non-spatial $T_1$-algebra, while $\F(\w L)$ is a compact $T_1$-algebra. We claim that $\nu$ is not induced by any MT-morphism $h:\F(\w L)\to \F L$.
    Indeed, if $\O(h)$ corresponded to $\nu$ under the isomorphisms $\O(\F(\w L))\cong \w L$ and $\O(\F L)\cong L$, then \cref{thm:O-compactifications-T1} would imply that $(\F(\w L),h)$ is a $T_1$-compactification of the non-spatial MT-algebra $\F L$. This contradicts \cref{cor:no-compactification}.
\end{*example}

Thus, assuming choice, $T_1$-compactifications in the sense of
\cref{def:compactification} are restricted to spatial MT-algebras. In
particular, one cannot expect Wallman or Stone--\v{C}ech compactifications
of arbitrary non-spatial MT-algebras in this sense. The spatial case is
different. In the next section we show, that every
spatial $T_1$ MT-algebra admits a Wallman compactification, and that every
spatial completely regular MT-algebra admits a Stone--\v{C}ech
compactification.

To end this section,
we generalize the well-known fact that the Alexandroff extension of a non-compact locally compact Hausdorff space is a Hausdorff compactification (see, e.g., \cite[p.~150]{Wil70}). We prove the MT-algebraic version without assuming choice, and hence without imposing spatiality. For this, we require the following result. 

\begin{lemma} \label{compact-eqv-sep}
    Let $M$ be a compact $T_1$-algebra. The following are equivalent.
    \begin{enumerate}[cref=lemma]
        \item $M$ is Hausdorff. \label{compact-eqv-T2}
        \item $M$ is regular. \label{compact-eqv-reg}
        \item $M$ is normal. \label{compact-eqv-norm}
    \end{enumerate}
\end{lemma}

\begin{proof}
    The implications \tref{compact-eqv-norm}$\Rightarrow$\tref{compact-eqv-reg}$\Rightarrow$\tref{compact-eqv-T2} are proven in \cite{BR23}, and the implication \tref{compact-eqv-T2}$\Rightarrow$\tref{compact-eqv-reg} in \cite{BR25}, so it suffices to show that \tref{compact-eqv-reg}$\Rightarrow$\tref{compact-eqv-norm}. Let $M$ be regular. Then $\O(M)$ is regular (see \cite[Sec.~7]{BR23}). Similarly, $\O(M)$ is compact because $M$ is. Thus, $\O(M)$ is a compact regular frame, and hence normal (see, e.g., \cite[Prop.~VII-2.2]{PP12}). Consequently, $M$ is normal by \cite[Thm.~8.15]{BR23}.
\end{proof}

\begin{remark}
The equivalence in \cref{compact-eqv-sep} does not extend to complete regularity 
without assuming choice.
The implication from normality to complete regularity proved in \cite[Thm.~8.17]{BR23} relies on the MT-algebraic version of Urysohn's Lemma, and hence on Countable Dependent Choice (see also \cite[p.~131--132]{PP21}).
\end{remark}

Recall that an MT-algebra is \emph{locally compact} provided
\[
    u=\bigvee\{v\in\O(M)\mid v\leq k\leq u \text{ for some compact } k\}
\]
for every $u\in\O(M)$.
\begin{theorem}
    $(\A(M),\pi)$ is a Hausdorff compactification iff $M$ is non-compact, locally compact, and Hausdorff.
\end{theorem}

\begin{proof}
    ($\Rightarrow$) Suppose $(\A(M),\pi)$ is a Hausdorff compactification of $M$. Since $\pi$ is dense, $M$ is not compact. Indeed, if $M$ were compact, then $\square(0,1)=(0,1)\neq(0,0)$,
    while $\pi(\square(0,1))=0$, contradicting density. $M$ is Hausdorff follows from \cref{prop:T2 under onto MTmorphism}.  It remains to show that $M$ is locally compact. Let $u\in\O(M)$. Since $\A(M)$ is compact Hausdorff, it is regular by \cref{compact-eqv-sep}. Hence,
    \[
        (u,0)=\bigvee\{(v,b)\in\O(\A(M))\mid (v,b)\prec(u,0)\}.
    \]
    If $(v,b)\prec(u,0)$, then $\diamond(v,b)\leq(u,0)$.
    By \cref{lem:alex-closed}, this implies that $b=0$, $\diamond v$ is compact, and $\diamond v\leq u$. Therefore,
    \[
        u=\bigvee\{v\in\O(M)\mid v\leq\diamond v\leq u
        \text{ and } \diamond v \text{ is compact}\}.
    \]
    Thus, $M$ is locally compact.

($\Leftarrow$) Suppose $M$ is non-compact, locally compact, and Hausdorff. By \cref{Alexandroff compactification}, $(\A(M),\pi)$ is a compactification of $M$. It remains to show that $\A(M)$ is Hausdorff. Let $(a,b) \in \A(M)$ be non-zero.
    Suppose first that $b=1$. Then $(0,1)\leq(a,b)$. Put
    \[
        c=\bigwedge\{\diamond u \mid u\in\O(M) \text{ and } \neg u \text{ is compact}\}.
    \]
    By \cref{lem:alex-open,lem:alex-closed},
    \[
        \bigwedge\{\diamond(u,v) \mid (u,v)\in\O(\A(M)) \text{ and } (0,1)\leq(u,v)\}
        =(c,1).
    \]
    We show that $c=0$. Suppose otherwise. By local compactness, there exist $w\in\O(M)$ and a compact element $k\in M$ such that
    $c\wedge w\neq0$ and $w \leq k$.
    Since $M$ is Hausdorff, $k$ is closed by \cite[Lem.~6.4]{BR25}. Thus, $\neg k\in\O(M)$ and $k=\neg(\neg k)$ is compact. Therefore, $c\leq\diamond\neg k$. On the other hand, since $w$ is open and $w\leq k$, we have $w\leq\square k=\neg\diamond\neg k$,
    contradicting $c\wedge w\neq0$. Hence, $c=0$, and so $(0,1)$ is a non-zero Hausdorff element below $(a,b)$.

    Suppose now that $b=0$. Then $a\neq0$. Since $M$ is Hausdorff, there exists a non-zero Hausdorff element $a'\leq a$, that is,
    \[
        a'=\bigwedge\{\diamond u \mid a'\leq u\in\O(M)\}.
    \]
    By local compactness, there exist $u\in\O(M)$ and a compact element $k\in M$ such that $a'\wedge u\neq0$ and $u\leq k$.
    Applying Hausdorffness again, we may replace $a'$ by a non-zero Hausdorff element below $a'\wedge u$. Thus, we may assume that
    \[
        0\neq a'\leq u\leq k\leq a.
    \]
    Since $M$ is Hausdorff, $k$ is closed by \cite[Lem.~6.4]{BR25}. Hence, $\diamond u\leq k$. Since $\diamond u$ is closed and lies below the compact element $k$, it is compact. Therefore, $\diamond(u,0)=(\diamond u,0)$.
    Since $(u,0)$ is an open element above $(a',0)$, and since every $(v,0)$ with $a'\leq v\in\O(M)$ is an open element above $(a',0)$, we obtain
    \[
        (a',0)
        =
        \bigwedge\{\diamond(v,d) \mid (v,d)\in\O(\A(M)) \text{ and } (a',0)\leq(v,d)\}.
    \]
    Thus, $(a',0)$ is a non-zero Hausdorff element below $(a,0)$. Hence, $\A(M)$ is Hausdorff.
\end{proof}
\begin{remark}
The previous theorem does not contradict \cref{cor:no-compactification}. Indeed, in the presence of choice, locally compact Hausdorff MT-algebras are spatial \cite[Lem.~6.3]{BR25}. 
\end{remark}

\section{Wallman-type compactifications of spatial MT-algebras}
\label{sec:wallman}

In the previous section, we saw that a non-spatial MT-algebra admits no $T_1$-compactifi\-cation. 
In this section, we construct a choice-free $T_1$-compactification for each spatial $T_1$-algebra.
Given a suitable basis of closed elements of an MT-algebra $M$, we first construct an associated compact $T_1$ MT-algebra. If $M$ is spatial and $T_1$, then this construction is a compactification of $M$. Taking the basis to be all closed elements yields the Wallman compactification, while, for spatial completely regular MT-algebras, taking the basis of zero-elements yields the Stone--\v{C}ech compactification. We then relate these constructions to the corresponding compactifications of frames of opens and, assuming choice, to the classical compactifications of spaces.

\subsection{Wallman extensions of MT-algebras}
\begin{definition}[Wallman bases]
Let $M$ be an MT-algebra and let $\B \subseteq \C(M)$. 
We consider the following conditions:
\begin{enumerate}[(B1)]
    \item $c = \bigwedge\{b \in \B \mid c \leq b\}$ for each $c \in \C(M)$; \label{cond-1}
    \item $\B$ is a bounded sublattice of $\C(M)$; \label{cond-2}
    \item $u = \bigvee \{b \in \B \mid b \leq u\}$ for each $u \in \O(M)$. \label{cond-3}
\end{enumerate}
We call $\B$ a \emph{closed basis} if it satisfies \ref{cond-1}, and a \emph{Wallman basis} if it additionally satisfies \ref{cond-2} and \ref{cond-3}. For $b \in \B$, we call $b$ \emph{$\B$-basic closed} and $\neg b$ \emph{$\B$-basic open}. We write $\BO$ for the set of $\B$-basic opens. When $\B$ is clear from context, we simply speak of \emph{basic closed} and \emph{basic open} elements.
\end{definition}

Let $M$ be an MT-algebra and $\B$ a bounded sublattice of $\C(M)$.\footnote{Every bounded distributive lattice $\B$ arises in this way. Indeed, if $\B$ is a bounded distributive lattice, then its boolean envelope embeds into a complete boolean algebra. Equipping this complete boolean algebra with the identity interior operator yields an MT-algebra $M$ such that $\B$ is isomorphic to a bounded sublattice of $\C(M)$.} Let $X_{\B}$ be the set of proper filters of $\B$, ordered by inclusion and equipped with the upset topology. 
Let $\RO(X_{\B})$ be the complete boolean algebra of regular opens of $X_{\B}$, that is, the upsets $U$ of $X_{\B}$ such that
\[
    U = \int \cl U = U^{**} = \{x \in X_{\B} \mid \upset x \subseteq \downset U\}.
\]
For $a \in M$, define
\[
    \sigma_{\B}(a)
    =
    \{x \in X_{\B} \mid 
       x \leq y \text{ implies }
       a \wedge b \neq 0 \text{ for all } b \in y\}.
\]
When $\B$ is clear from context, we write $\sigma(a)$. We first observe that $\sigma(a) \in \RO(X_{\B})$ for each $a \in M$. Indeed, suppose $x \in \sigma(a)$ and $x \leq z$. If $z \leq y$ and $b \in y$, then $x \leq y$, so $a \wedge b \neq 0$. Hence, $z \in \sigma(a)$, and therefore $\sigma(a)$ is an upset. Now suppose $x \in \sigma(a)^{**}$, so $\upset x \subseteq \downset \sigma(a)$. Let $x \leq y$ and $b \in y$. Since $y \in \upset x$, there exists $z \in \sigma(a)$ such that $y \leq z$. Thus, $b \in z$, and hence $a \wedge b \neq 0$. Therefore, $x \in \sigma(a)$, proving that $\sigma(a)^{**} \subseteq \sigma(a)$. Consequently, $\sigma$ defines a map
$\sigma : M \to \RO(X_{\B})$.

\begin{remark}
This construction is related to the use of spaces of proper filters and regular opens in choice-free Stone duality \cite{BH20}.
\end{remark}

The following lemma describes the elements of $\sigma(a)$ and gives a useful sufficient condition for membership. For $S \subseteq \B$, we write $\<S\rangle$ for the filter of $\B$ generated by $S$.

\needspace{2em}
\begin{lemma}
Let $a \in M$ and $x \in X_{\B}$. 
\begin{enumerate}[cref=lemma]
    \item $x \in \sigma(a)$ iff for each $b \in \B$ with $a \wedge b = 0$, there exists $c \in x$ such that $c \wedge b = 0$.\label{lem:sigma-def}

    \item If there exists $c \in x$ such that $c \leq a$, then $x \in \sigma(a)$. In particular, if $a \in \B$ and $a \in x$, then $x \in \sigma(a)$.\label{lem:inside-lemma}
\end{enumerate}
\end{lemma}

\begin{proof}
\tref{lem:sigma-def}
Suppose $x \in \sigma(a)$ and let $b \in \B$ satisfy $a \wedge b = 0$. Consider $y = \< x \cup \{b\}\rangle$. If $y$ were proper, then $x \leq y$ and $b \in y$, contradicting $x \in \sigma(a)$. Thus, there exists $c \in x$ such that $c \wedge b = 0$.
Conversely, suppose the stated condition holds. Let $y \in X_{\B}$ with $x \leq y$ and let $d \in y$. If $a \wedge d = 0$, then there exists $c \in x$ such that $c \wedge d = 0$. Since $c,d \in y$, this contradicts the fact that $y$ is proper. Hence, $a \wedge d \neq 0$, and therefore $x \in \sigma(a)$.

\tref{lem:inside-lemma}
Suppose $c \in x$ and $c \leq a$. If $b \in \B$ satisfies $a \wedge b = 0$, then $c \wedge b = 0$. Thus, $x \in \sigma(a)$ by \tref{lem:sigma-def}.
\end{proof}

We next record the properties of $\sigma$ that will be needed below. Recall that joins and meets in $\RO(X_{\B})$ are given by
\[
    A \vee B = (A \cup B)^{**}
    \qquad\text{and}\qquad
    A \wedge B = A \cap B.
\]
In particular, the following lemma shows that $\sigma$ preserves finite joins and preserves the boolean operations on the basic open and closed elements.

\begin{lemma}\label{sigma-properties}
\leavevmode
    Let $a,b \in M$.
    \begin{enumerate}[cref=lemma]
        \item $a \leq b$ implies $\sigma(a) \subseteq \sigma(b)$, i.e., $\sigma$ is monotone;\label{lem:sigma-order-preserving}

        \item $\sigma(a \vee b) = \sigma(a) \vee \sigma(b)$;\label{lem:opens-RO-joins}
        
        \item If $a \in \BC \cup \BO$, then $\sigma(\neg a) = \sigma(a)^*$;\label{lem:sigma-neg}  
        
        \item If $a,b \in \BC \cup \BO$, then $\sigma(a \wedge b) = \sigma(a) \wedge \sigma(b)$;\label{lem:sigma-meets}\label{lem:closed-RO-meets}
        
        \item If $a,b \in \BC \cup \BO$, then $a \leq b$ iff $\sigma(a) \subseteq \sigma(b)$. \label{lem:injective-basics}
    \end{enumerate}
\end{lemma}
\begin{proof}
    \tref{lem:sigma-order-preserving} Let $x \in \sigma(a)$. Suppose $x \leq y$ and $c \in y$. Then $c \wedge a \neq 0$ since $x \in \sigma(a)$. Therefore, $c \wedge b \neq 0$ since $a \leq b$. Thus, $x \in \sigma(b)$.

    \tref{lem:opens-RO-joins} By \tref{lem:sigma-order-preserving}, $\sigma(a) \subseteq \sigma(a \vee b)$ and $\sigma(b) \subseteq \sigma(a \vee b)$. Since $\sigma(a \vee b)$ is regular open, it follows that $\sigma(a) \vee \sigma(b) \subseteq \sigma(a \vee b)$. Conversely, suppose $x \in \sigma(a \vee b)$. Let $x \leq y$. If $y \in \sigma(a)$, there is nothing to prove. Otherwise, by \cref{lem:sigma-def}, there exists $c \in \B$ such that $c \wedge a = 0$ and $e \wedge c \neq 0$ for each $e \in y$. Hence $z := \< y \cup {c} \rangle$ is a proper filter with $y \leq z$. We claim that $z \in \sigma(b)$. If not, then again by \cref{lem:sigma-def}, there exists $d \in \B$ such that $d \wedge b = 0$ and $e \wedge d \neq 0$ for each $e \in z$. Thus $z' := \< z \cup {d} \rangle$ is a proper filter. Since $x \leq y \leq z \leq z'$ and $x \in \sigma(a \vee b)$, we have $z' \in \sigma(a \vee b)$. But $c,d \in z'$, so $c \wedge d \in z'$, while $(c \wedge d) \wedge (a \vee b) = (c \wedge d \wedge a) \vee (c \wedge d \wedge b) = 0$, contradicting $z' \in \sigma(a \vee b)$. Therefore $z \in \sigma(b)$, and since $y \leq z$, we get $y \in \downset(\sigma(a) \cup \sigma(b))$. Thus $\upset x \subseteq \downset(\sigma(a) \cup \sigma(b))$, so $x \in (\sigma(a) \cup \sigma(b))^{**} = \sigma(a) \vee \sigma(b)$.

    \tref{lem:sigma-neg} Suppose $x \in \sigma(\neg a)$ and $y \in \upset x \cap \sigma(a)$. If $a \in \BC$, then $\< y \cup \{a\}\rangle$ is a proper filter extending $x$ and containing $a$, contradicting $x \in \sigma(\neg a)$. If $a \in \BO$, then $\neg a \in \BC$ and $\< y \cup \{\neg a\}\rangle$ is a proper filter extending $y$ and containing $\neg a$, contradicting $y \in \sigma(a)$. Thus, $\sigma(\neg a) \subseteq \sigma(a)^*$. Conversely, if $x \in \sigma(a)^*$ and $y \geq x$, then $y \notin \sigma(a)$, so there exist $z \geq y$ and $c \in z$ such that $c \wedge a = 0$. Since $c \leq \neg a$, \cref{lem:inside-lemma} yields $z \in \sigma(\neg a)$. Hence, $x \in \sigma(\neg a)^{**}=\sigma(\neg a)$.

    \tref{lem:sigma-meets} By \tref{lem:sigma-order-preserving}, we have
    $\sigma(a \wedge b) \subseteq \sigma(a) \wedge \sigma(b)$. For the reverse
    inclusion, first suppose that $a \in \BC$. Let
    $x \in \sigma(a) \cap \sigma(b)$, let $x \leq y$, and let $c \in y$.
    Since $x \in \sigma(a)$, the filter $z=\< y \cup \{a\}\rangle$ is
    proper. Since $x \leq z$ and $x \in \sigma(b)$, we have
    $c \wedge a \wedge b \neq 0$. Thus, $x \in \sigma(a \wedge b)$.
    The same argument applies if $b \in \BC$. It remains to consider the
    case $a,b \in \BO$. Then $\neg a \vee \neg b \in \BC$, and hence, by
    \tref{lem:opens-RO-joins} and \tref{lem:sigma-neg},
    \[
        \sigma(a \wedge b)
        = \sigma(\neg(\neg a \vee \neg b))
        = \sigma(\neg a \vee \neg b)^*
        = \sigma(a) \wedge \sigma(b).
    \]

    \tref{lem:injective-basics} The left-to-right implication follows from \tref{lem:sigma-order-preserving}. Conversely, suppose $a \nleq b$. Since $a,b \in \BC \cup \BO$, there exists a non-zero $c \in \BC$ such that $c \leq a \wedge \neg b$: this is immediate if $a \wedge \neg b \in \BC$, and otherwise follows from \ref{cond-3}. Then $\< c\rangle \in \sigma(a)$ by \tref{lem:inside-lemma}, while $\< c\rangle \notin \sigma(b)$ since $c \wedge b=0$. Thus, $\sigma(a) \nsubseteq \sigma(b)$.
\end{proof}

We view $\RO(X_{\B})$ as an MT-algebra by declaring its closed elements to be the arbitrary meets of elements of $\sigma[\B]$. By \cref{lem:sigma-neg}, its open elements are therefore the joins of elements $\sigma(u)$, where $u$ is basic open.

\begin{theorem}\label{thm:WM-compact}
Let $M$ be an MT-algebra and $\B$ a bounded sublattice of $\C(M)$. Then 
$\RO(X_{\B})$ is compact and $T_1$.
\end{theorem}
\begin{proof}
Let $A \in \RO(X_{\B})$ be non-zero and suppose $x \in A$. Put
$C = \bigcap\{\sigma(b) \mid b \in x\}$.
Then $C$ is closed and $x \in C$ by \cref{lem:inside-lemma}. We show that $C \subseteq A$. Let $y \in C$ and $y \leq z$. Since $C$ is an upset, $z \in C$. Hence, $z \in \sigma(b)$ for every $b \in x$, so $\< x \cup z\rangle$ is a proper filter of $\B$. Since it extends $x$ and $A$ is an upset, we have $\< x \cup z\rangle \in A$. Since it also extends $z$, we obtain $z \in \downset A$. Thus, $\upset y \subseteq \downset A$, and hence $y \in A^{**}=A$. Therefore, every non-zero element of $\RO(X_{\B})$ contains a non-zero closed element, so $\RO(X_{\B})$ is $T_1$.

For compactness, since the closed elements are meet-generated by $\sigma[\B]$, it suffices to consider a family $\mathcal C \subseteq \B$ such that $\bigcap\{\sigma(c) \mid c \in \mathcal C'\} \neq \varnothing$
for every finite $\mathcal C' \subseteq \mathcal C$. If $c_1,\dots,c_n \in \mathcal C$, then $\sigma(c_1) \cap \cdots \cap \sigma(c_n)
    = \sigma(c_1 \wedge \cdots \wedge c_n) \neq \varnothing$,
so $c_1 \wedge \cdots \wedge c_n \neq 0$ (because no filter can consist of elements disjoint from $0$). Thus, $\<\mathcal C\rangle$ is a proper filter of $\B$. By \cref{lem:inside-lemma}, $\<\mathcal C\rangle \in \bigcap\{\sigma(c) \mid c \in \mathcal C\}$.
Therefore, $\RO(X_{\B})$ is compact.
\end{proof}

\begin{theorem}\label{prop:WM-unique}
Let $M$ be a compact $T_1$-algebra and $\B$ a Wallman basis of $M$. Then $M \cong \RO(X_{\B})$.
\end{theorem}
\begin{proof}
Let $N=\RO(X_{\B})$. We first show that, for every $S\subseteq\B$, $\sigma\Bigl(\bigwedge S\Bigr)=\bigcap \sigma[S]$. By monotonicity, the left-hand side is contained in the right-hand side. Conversely, suppose $x\notin\sigma(\bigwedge S)$. Then there exist $y\geq x$ and $d\in y$ such that $d\wedge\bigwedge S=0$. Since $d$ is closed and $M$ is compact, $d$ is compact by \cref{closed-below-compact}. Hence, there is a finite $S'\subseteq S$ such that $d\wedge\bigwedge S'=0$ by \cref{compactness-meets}. It follows that $x\notin\sigma(\bigwedge S')=\bigcap\sigma[S']$, where the equality follows from \cref{lem:closed-RO-meets}. Thus, $x\notin\bigcap\sigma[S]$, proving the claim.

For $c\in\C(M)$, put $S_c=\{b\in\B\mid c\leq b\}$. By \ref{cond-1}, $c=\bigwedge S_c$, and hence the claim yields $\sigma(c)=\bigcap\sigma[S_c]\in\C(N)$. Conversely, every closed element of $N$ is of the form $\bigcap\sigma[S]$ for some $S\subseteq\B$, and hence is equal to $\sigma(\bigwedge S)$. Thus, $\sigma:\C(M)\to\C(N)$ is onto.

It remains to show that $\sigma$ reflects order on closed elements. Suppose $c,d\in\C(M)$ and $c\nleq d$. By \ref{cond-1}, there exists $b\in\B$ such that $d\leq b$ and $c\nleq b$. Therefore, $c\wedge\neg b\neq0$. Since $\neg b$ is open, \ref{cond-3} yields $e\in\B$ such that $e\leq\neg b$ and $c\wedge e\neq0$. Let $x=\<S_c\cup\{e\}\rangle$. Then $x$ is a proper filter, since $c\leq f$ for every $f\in S_c$ and $c\wedge e\neq0$. Moreover, $x\in\sigma(c)=\bigcap\sigma[S_c]$, since $S_c\subseteq x$. On the other hand, $e\in x$ and $e\wedge b=0$, so $x\notin\sigma(b)$. Since $d\leq b$, we have $\sigma(d)\subseteq\sigma(b)$, and hence $x\notin\sigma(d)$. Therefore, $\sigma(c)\nsubseteq\sigma(d)$.

Thus, $\sigma:\C(M)\to\C(N)$ is an order isomorphism. Since both $M$ and $N$ are $T_1$-algebras, it follows from \cref{lem:T1-iso} that $M\cong N=\RO(X_{\B})$.
\end{proof}

\begin{lemma}\label{lem:sigma-atoms}
    Let $M$ be a $T_1$-algebra and let $\B$ be a Wallman basis of $M$. Then $\sigma_{\B}(a)$ is an atom of $\RO(X_{\B})$ for each $a\in\at(M)$.
\end{lemma}

\begin{proof}
    For $a\in\at(M)$, let $x_a = \{b\in\B\mid a\leq b\}$.
    Then $x_a$ is a proper filter of $\B$. We first show that $x_a\in\sigma_{\B}(a)$. Let $b\in\B$ such that $a\wedge b=0$. Then $a\leq\neg b\in\O(M)$. By \ref{cond-3} and since $a$ is an atom, there exists $c\in\B$ such that $a\leq c\leq\neg b$. Thus $c\in x_a$ and $c\wedge b=0$, so
    $x_a\in\sigma_{\B}(a)$ by \cref{lem:sigma-def}. We claim that, for every $A\in\RO(X_{\B})$, $\sigma_{\B}(a)\subseteq A$ iff
        $x_a\in A$.
    The left-to-right implication follows from $x_a\in\sigma_{\B}(a)$.
    Conversely, suppose $x_a\in A$ and let $x\in\sigma_{\B}(a)$. If
    $x\leq y$, then $a\wedge b\neq 0$ for every $b\in y$, and hence
    $a\leq b$ for every $b\in y$. Thus $y\leq x_a$, so
    $y\in\downset A$. Therefore, $\uparrow x\subseteq\downarrow A$, and hence $x\in A^{**}=A$. This proves the claim.

    Finally, let $A\subseteq\sigma_{\B}(a)$ be non-zero with
    $A\in\RO(X_{\B})$. Suppose $x\in A$. Since $x\in\sigma_{\B}(a)$,
    we have $x\leq x_a$, so $x_a\in A$. By the claim,
    $\sigma_{\B}(a)\subseteq A$. Thus $A=\sigma_{\B}(a)$, proving that
    $\sigma_{\B}(a)$ is an atom.
\end{proof}

Let $M$ be a spatial $T_1$-algebra and let $\B$ be a Wallman basis of
    $M$. We define a map $\rho_{\B}:\RO(X_{\B})\to M$ by
    \[
        \rho_{\B}(A)
        =
        \bigvee\{a\in\at(M)\mid \sigma_{\B}(a)\subseteq A\}
    \]
    for every $A \in \RO(X_{\B})$.

\begin{lemma}\label{lem:rho-wallman-basis}
    The map $\rho_{\B}:\RO(X_{\B})\to M$ is a complete boolean
    homomorphism such that $\rho_{\B}(\sigma_{\B}(b))=b$ and 
    $\rho_{\B}(\sigma_{\B}(\neg b))=\neg b$ for each $b\in\B$.
\end{lemma}

\begin{proof}
    Write $\sigma=\sigma_{\B}$ and $\rho=\rho_{\B}$. For
    $a\in\at(M)$ and $A\in\RO(X_{\B})$, it is clear that $a\leq\rho(A)$ iff $\sigma(a)\subseteq A$.
    By \cref{lem:sigma-atoms}, $\sigma(a)$ is an atom of
    $\RO(X_{\B})$. Hence, for $\{A_i\}_{i\in I}\subseteq\RO(X_{\B})$,
    \[
        a\leq\rho\left(\bigcap A_i\right)
        \iff
        a\leq\bigwedge \rho(A_i),
    \]
    and, for $A\in\RO(X_{\B})$,
    \[
        a\leq\rho(A^*)
        \iff
        a\leq\neg\rho(A).
    \]
    Since $M$ is spatial, it follows that $\rho(\bigwedge A_i) = \bigwedge_{i\in I}\rho(A_i)$
        and  $\rho(A^*)=\neg\rho(A)$.
    Thus, $\rho$ is a complete boolean homomorphism.

    Let $b\in\B$ and $a\in\at(M)$. If $a\leq b$, then
    $\sigma(a)\subseteq\sigma(b)$ by monotonicity. Conversely, suppose
    $a\nleq b$. Since $a$ is an atom, $a\wedge b=0$. As in the proof of
    \cref{lem:sigma-atoms}, there is $c\in\B$ such that
    $a\leq c\leq\neg b$. Hence $c\in x_a\in\sigma(a)$, whereas
    $x_a\notin\sigma(b)$. Thus $\sigma(a)\nsubseteq\sigma(b)$.
    Therefore, $a\leq\rho(\sigma(b))$ iff
        $a\leq b$.
    Since $M$ is spatial, $\rho(\sigma(b))=b$. Finally, by
    \cref{lem:sigma-neg}, $\rho(\sigma(\neg b))
        =
        \rho(\sigma(b)^*)
        =
        \neg\rho(\sigma(b))
        =
        \neg b$.
\end{proof}

\begin{theorem}[Choice-free Wallman compactification]
    \label{thm:choice-free-wallman}
    Let $M$ be a spatial $T_1$-algebra and let $\B$ be a Wallman basis of
    $M$. Then $(\RO(X_{\B}),\rho_{\B})$ is a $T_1$-compactification of
    $M$.
\end{theorem}

\begin{proof}
 Set $N=\RO(X_{\B})$, $\sigma=\sigma_{\B}$, and $\rho=\rho_{\B}$.
    By \cref{thm:WM-compact}, $N$ is compact and $T_1$, and by
    \cref{lem:rho-wallman-basis}, $\rho$ is a complete boolean
    homomorphism.

    Since the closed elements of $N$ are meet-generated by the elements
    $\sigma(b)$, with $b\in\B$, and
    $\rho(\sigma(b))=b\in \B$, the map $\rho$ sends closed elements to
    closed elements. Hence, for every $A\in N$,
    $\diamond\rho(A)\leq\rho(\diamond A)$,
    so $\rho$ is an MT-morphism.

    Let $c\in\C(M)$. Then $c=\bigwedge\{b\in\B\mid c\leq b\}$, 
    so 
        ${c
        =
        \rho(
            \bigcap\{\sigma(b)\mid b\in\B,\ c\leq b\})}$.
    Since $M$ is $T_1$, its closed elements join-generate $M$, so $\rho$
    is onto. If $u\in\O(M)$, apply the preceding identity to
    $c=\neg u$. Taking complements gives an open element of $N$ mapped
    to $u$. Thus, $\O(\rho)$ is onto.

    Finally, let $U\in\O(N)$ and suppose $\rho(U)=0$. Since the elements
    $\sigma(\neg b)$, with $b\in\B$, join-generate $\O(N)$,
    \[
        U
        =
        \bigvee\{\sigma(\neg b)\mid b\in\B,\
        \sigma(\neg b)\subseteq U\}.
    \]
    For each $b\in\B$ with $\sigma(\neg b)\subseteq U$,
    \[
        \neg b
        =
        \rho(\sigma(\neg b))
        \leq
        \rho(U)
        =
        0.
    \]
    Hence $\sigma(\neg b)=0$, and therefore $U=0$. Thus, $\rho$ is
    dense.
\end{proof}

Taking $\B=\C(M)$ yields a compact $T_1$-algebra which we denote by $\w M=\RO(X_{\C(M)})$ and call the \emph{Wallman extension} of $M$.

\begin{corollary}
    If $M$ is a spatial $T_1$-algebra, then $(\w M, \rho_{\C(M)})$ is a $T_1$-compactification of $M$.
\end{corollary}

To relate the Wallman extension of MT-algebras to the classical construction, let us return to $\RO(X_{\B})$, where $\B$ is a bounded sublattice of $\C(M)$. Assuming choice, this algebra admits the usual description in terms of maximal filters. For a poset $X$, let $\max X$ denote its set of maximal elements. Thus, $\max X_{\B}$ is the set of maximal proper filters of $\B$. We equip $\max X_{\B}$ with the topology having $\{\max \sigma(b) \mid b \in \B\}$ as a closed base.

\begin{*proposition}\label{prop:RO-max}
$\RO(X_{\B}) \cong \P(\max X_{\B})$.
\end{*proposition}

\begin{proof}
We claim the map $\max : \RO(X_{\B}) \to \P(\max X_{\B})$, given by $U \mapsto \max U$, is an isomorphism of MT-algebras. By Zorn's Lemma, every proper filter of $\B$ is contained in a maximal proper filter. Thus, for each $U \in \RO(X_{\B})$ and $x \in X_{\B}$ we have $x \in U$  iff $\max(\upset x) \subseteq U$. Indeed, the left-to-right implication follows since $U$ is an upset. Conversely, if $\max(\upset x) \subseteq U$ and $y \geq x$, then $y$ is contained in some maximal proper filter $z \geq x$. Thus, $z \in U$, so $y \in \downset U$. Hence, $\upset x \subseteq \downset U$, and therefore $x \in U^{**}=U$. It follows that $\max$ is injective. 

To see that it is onto, for $S \subseteq \max X_{\B}$, put $U_S=S^{**}$.
Then $U_S$ is an upset and 
\[
U_S = \{x \in X_{\B} \mid \upset x \subseteq \downset S \} = \{x \in X_{\B} \mid \max(\upset x) \subseteq S\},
\]
so $\max U_S=S$. Therefore, $\max$ is onto, and hence, since it is order-preserving, it is an order-isomorphism. Thus, $\max$ is a complete boolean isomorphism. Moreover, for each $b \in \B$, we have $\max\sigma(b)=\{x \in \max X_{\B} \mid b \in x\}$. Indeed, if $b \in x$, then $x \in \sigma(b)$ by \cref{lem:inside-lemma}. Conversely, if $b \notin x$, then maximality of $x$ implies that there is $c \in x$ such that $b \wedge c=0$, so $x \notin \sigma(b)$. Since the closed elements of $\RO(X_{\B})$ are meet-generated by the elements $\sigma(b)$, and the closed elements of $\P(\max X_{\B})$ are meet-generated by the sets $\max\sigma(b)$, it follows that $\max$ is an isomorphism.
\end{proof}

We now relate the Wallman extension of MT-algebras to the classical Wallman extension of spaces. Recall that, for a $T_1$-space $X$, the Wallman extension $\w X$ is the space $\max X_{\C(\P(X))}$ (see \cref{rem:wallman}) and the pair $(\w X, e)$ is the \emph{Wallman compactification} of $X$, where $e : X \to \w X$ is defined by $e(x) = \{C \in \C(\P(X)) \mid x \in C\}$.

\begin{*theorem} \label{thm:wallman-spc}
Let $X$ be a $T_1$-space. Then $\P(\w X) \cong \w\P(X)$. Consequently, $\w X \cong \at(\w\P(X))$.
\end{*theorem}

\begin{proof}
By definition, $\w\P(X)=\RO(X_{\C(\P(X))})$. Hence, by \cref{prop:RO-max},
$\w\P(X)\cong\P(\max X_{\C(\P(X))})=\P(\w X)$ as MT-algebras. Applying $\at$ yields $\at(\w\P(X))\cong\w X$.
\end{proof}

Recall from \cref{thm:frm-wallman} that, for a subfit frame $L$, the Wallman compactification is
\[
    \nu : \w L=\J(L)_s \longrightarrow L,
    \qquad
    \nu(I)=\bigvee I.
\]
We now show that the Wallman extension of a $T_1$-algebra recovers this construction on its frame of opens.

\begin{theorem}\label{prop:subfit-wallmann}
Let $M$ be a $T_1$-algebra. Define $\rho : \O(\w M) \to \O(M)$ by
\[
    \rho(U)=\bigvee\{u\in\O(M)\mid \sigma(u)\subseteq U\}.
\]
Then there is a frame isomorphism $\alpha : \O(\w M) \to \w \O(M)$
such that the following diagram commutes:
\[
    \begin{tikzcd}
    \O(\w M) \ar[r, "\alpha", "\cong"'] \ar[rd, "\rho"', bend right=.5cm]
        & \w \O(M) \ar[d, "\nu"] \\
        & \O(M).
    \end{tikzcd}
\]
In particular, $(\O(\w M),\rho)$ is the Wallman compactification of $\O(M)$.
\end{theorem}
\begin{proof}
Since $M$ is $T_1$, the frame $\O(M)$ is subfit and $\C(M)$ is a Wallman basis. Define
\[
    \alpha : \O(\w M) \to \w \O(M),
    \qquad
    \alpha(U)=\{u \in \O(M) \mid \sigma(u)\subseteq U\},
\]
and
\[
    \beta : \w \O(M) \to \O(\w M),
    \qquad
    \beta(I)=\bigvee\{\sigma(u) \mid u \in I\}.
\]
Since every open element of $M$ is $\C(M)$-basic open, $\beta$ is well defined. By \cref{lem:opens-RO-joins}, $\alpha(U)$ is an ideal of $\O(M)$.

We show that $\alpha(U)$ is suited. Let $I$ be an $\alpha(U)$-small ideal and let $v\in I$. Suppose $x\in\sigma(v)$ and let $y\geq x$. Then $y\in\sigma(v)$, so $y\notin\sigma(\neg v)$ by \cref{lem:sigma-neg}. Hence, there exist $z\geq y$ and $c\in z$ such that $c\wedge\neg v=0$. Thus, $v\vee\neg c=1$. Let $J$ be the ideal generated by $\neg c$. Since $v\in I$, we have $I\vee J=\O(M)$, and therefore $\alpha(U)\vee J=\O(M)$. Hence, there exists $u\in\alpha(U)$ such that $u\vee\neg c=1$, so $c\leq u$. By \cref{lem:inside-lemma}, $z\in\sigma(u)\subseteq U$. Therefore, $y\in\downset U$. Since this holds for each $y\geq x$, we have $x\in U^{**}=U$. Thus, $\sigma(v)\subseteq U$, so $v\in\alpha(U)$. Hence, every $\alpha(U)$-small ideal is contained in $\alpha(U)$, and therefore $\alpha(U)$ is suited.

It is immediate that $\alpha$ and $\beta$ are order-preserving. If $U\in\O(\w M)$, then $\beta(\alpha(U))=U$, since the open elements of $\w M$ are joins of elements $\sigma(u)$ with $u\in\O(M)$. Now let $I\in\w \O(M)$. Clearly, $I\subseteq\alpha(\beta(I))$. Conversely, suppose $u\in\alpha(\beta(I))$, and let $J$ be the principal ideal generated by $u$. We show that $J$ is $I$-small. Suppose $J\vee K=\O(M)$. Then $u\vee v=1$ for some $v\in K$, so $\neg v\leq u$ and hence
$\sigma(\neg v)\subseteq\sigma(u)\subseteq\beta(I)$.
Since $\sigma(\neg v)$ is closed and $\w M$ is compact, there exist $u_1,\dots,u_n\in I$ such that $\sigma(\neg v)\subseteq\sigma(u_1\vee\cdots\vee u_n)$. By \cref{lem:injective-basics}, $\neg v\leq u_1\vee\cdots\vee u_n$, so $u_1\vee\cdots\vee u_n\vee v=1$. Thus, $I\vee K=\O(M)$, and hence $J$ is $I$-small. Since $I$ is suited, $u\in I$. Therefore, $\alpha(\beta(I))=I$, and $\alpha$ is an isomorphism with inverse $\beta$.

Finally, for each $U\in\O(\w M)$,
\[
    (\nu\circ\alpha)(U)
    = \bigvee\{u\in\O(M)\mid \sigma(u)\subseteq U\}
    = \rho(U).
\]
Therefore, $\nu\circ\alpha=\rho$, and the displayed diagram commutes. Since $\alpha$ is an isomorphism and $\nu:\w \O(M)\to\O(M)$ is the Wallman compactification of $\O(M)$, so is $\rho:\O(\w M)\to\O(M)$.
\end{proof}

\begin{corollary}
Let $L$ be a subfit frame. Then there exists a $T_1$-algebra $M$ such that the Wallman compactification of $L$ is $(\O(\w M),\rho)$.
\end{corollary}

\begin{proof}
Take $M=\F L$. Then $\O(M)\cong L$, and $M$ is $T_1$ by \cref{prop:subfit-iff-t1}. The result follows from \cref{prop:subfit-wallmann}.
\end{proof}

\begin{remark}
    Let $M$ be a $T_1$-algebra. Then, since $\w M$ is $T_1$, \cref{lem:T1-iso,prop:subfit-wallmann} yield
    \[
        \w M \cong \F\O(\w M) \cong \F(\w\O(M)).
    \]
    Thus, the Wallman extension of a $T_1$-algebra may equivalently be obtained by taking the Funayama envelope of the Wallman compactification of its frame of opens.
\end{remark}

\subsection{The Stone-\v{C}ech extension of an MT-algebra}

In this subsection, we use the Wallman bases for MT-algebras to obtain an analogue of the Stone--\v{C}ech extension. This requires the following definition.

\begin{definition}
Let $M$ be an MT-algebra and let $a \in M$.
\begin{enumerate}
    \item We call $a$ a \emph{zero-element} if there exists a sequence $(c_n)_{n\in\mathbb N}$ in $\C(M)$ such that $a=\bigwedge_{n\in\mathbb N}c_n$ and $c_{n+1}\pprec c_n$ for each $n\in\mathbb N$.
    \item We call $a$ a \emph{cozero-element} if $\neg a$ is a zero-element.
\end{enumerate}
We write $\Z(M)$ and $\Coz(M)$ for the sets of zero-elements and cozero-elements of $M$, respectively.
\end{definition}

Cozero-elements are open. Moreover, since $a \pprec b$ implies $\neg b \pprec \neg a$, an element $u \in M$ is a cozero-element iff there exists a sequence $(u_n)_{n\in\mathbb N}$ in $\O(M)$ such that $u=\bigvee_{n\in\mathbb N}u_n$ and $u_n\pprec u_{n+1}$ for each $n\in\mathbb N$. Thus, the cozero-elements of $M$ are precisely the cozero-elements of the frame $\O(M)$ in the usual frame-theoretic sense; see \cite[p.~286]{PP12}. Thus, $\Z(M)$ is a bounded sublattice of $\C(M)$ since $\Coz(M)$ is a bounded sublattice of $\O(M)$. 

\begin{lemma} \label{zero-between}
    Let $M$ be an MT-algebra. If $a \pprec b$, then there exists $z \in \Z(M)$ such that $a \leq z \leq b$.
\end{lemma}

\begin{proof}
    If $a \pprec b$, then there is a family $\{c_p\mid p\in[0,1]\cap\mathbb Q\}\subseteq\C(M)$ such that $a\leq c_0$, $c_1\leq b$, and $c_p\pprec c_q$ whenever $p<q$. Then $z=\bigwedge_{n\geq 1}c_{2^{-n}}$ is a zero-element and $a\leq z\leq b$.
\end{proof}

\begin{lemma}\label{lem:zero-wallman-basis}
Let $M$ be a $T_1$-algebra. Then $M$ is completely regular iff $\Z(M)$ is a Wallman basis.
\end{lemma}

\begin{proof}
($\Rightarrow$) Suppose $M$ is completely regular. Then $\O(M)$ is completely regular, so $\Coz(M)$ join-generates $\O(M)$ (see, e.g., \cite[p.~286]{PP12}). Thus, every closed element of $M$ is a meet of zero-elements above it. Hence, $\Z(M)$ satisfies \ref{cond-1}. It remains to verify \ref{cond-3}. Let $u\in\O(M)$. By complete regularity, $u=\bigvee\{v\in\O(M)\mid v\pprec u\}$. If $v \pprec u$, then by \cref{zero-between}, there is $z \in \Z(M)$ with $v\leq z\leq u$. Hence, $u$ is a join of zero-elements below it, and \ref{cond-3} holds.

($\Leftarrow$) Conversely, suppose $\Z(M)$ is a Wallman basis. By \ref{cond-1}, every closed element is a meet of zero-elements above it. Taking complements, every open element is a join of cozero-elements. Hence, $\Coz(M)$ join-generates $\O(M)$, so $\O(M)$ is completely regular. Since $M$ is $T_1$, it follows that $M$ is completely regular.
\end{proof}

The next theorem characterizes the Wallman bases that give rise to Hausdorff MT-algebras.

\begin{theorem}\label{lem:Wallman-normality}
    Let $\B$ be a Wallman basis of an MT-algebra $M$. Then $\RO(X_{\B})$ is Hausdorff iff for all disjoint basic closed elements $c,d\in\B$ there exist disjoint basic open elements $u,v$ such that $c\leq u$ and $d\leq v$.
\end{theorem}

\begin{proof}
    By \cref{compact-eqv-sep}, $\RO(X_{\B})$ is Hausdorff iff $\RO(X_{\B})$ is normal.

    ($\Rightarrow$) Suppose $c, d \in \BC$ and $c \wedge d = 0$. Then $\sigma(c) \cap \sigma(d) = \varnothing$ by \cref{lem:sigma-meets}, so by normality of $\RO(X_{\mathcal B})$, there exists disjoint opens $U,V$ such that $\sigma(c) \subseteq U$ and $\sigma(d) \subseteq V$. By compactness, there exist $u,v \in \BO$ such that $\sigma(c) \subseteq \sigma(u) \subseteq U$ and $\sigma(d) \subseteq \sigma(v) \subseteq V$. By \cref{lem:injective-basics}, $c \leq u$ and $d \leq v$ while $u \wedge v = 0$, as required.

    ($\Leftarrow)$ Suppose $C \cap D = \varnothing$. By compactness and \cref{lem:closed-RO-meets}, there exist disjoint basic closed $c, d \in \mathcal B$ such that $C \subseteq \sigma(c)$ and $D \subseteq \sigma(d)$. Therefore, there are disjoint basic opens $u,v$ such that $c \leq u$ and $d \leq v$. Then $\sigma(u) \cap \sigma(v) = \sigma(u \wedge v) = \varnothing $ by \cref{lem:sigma-meets}, and by \cref{lem:sigma-order-preserving}, $C \subseteq \sigma(c) \subseteq \sigma(u)$ and $D \subseteq \sigma(d) \subseteq \sigma(v)$, so $\RO(X_{\mathcal B})$ is normal.
\end{proof}

\begin{corollary}\label{prop:W-hausdorff}
Let $M$ be a $T_1$-algebra. Then $\w M$ is Hausdorff iff $M$ is normal.
\end{corollary}
\begin{proof}
    If $M$ is $T_1$, then $\C(M)$ is a Wallman basis, so \cref{lem:Wallman-normality} applies. 
\end{proof}

For a completely regular MT-algebra $M$, we call $\beta M=\RO(X_{\Z(M)})$ the \emph{Stone--\v{C}ech extension} of $M$. 

\begin{theorem}\label{prop:beta-hausdorff}
    Let $M$ be a completely regular MT-algebra. Then $\beta M$ is compact Hausdorff, and if $M$ is compact, then $M \cong \beta M$.
\end{theorem}

\begin{proof}
By \cref{lem:zero-wallman-basis}, $\Z(M)$ is a Wallman basis. Thus, \cref{thm:WM-compact} yields that $\beta M$ is compact $T_1$. Let $z_1,z_2\in\Z(M)$ be disjoint, and put $u_i=\neg z_i$ for $i=1,2$. Then $u_1,u_2\in\Coz(M)$ and $u_1\vee u_2=1$. By \cite[Cor.~5.1.3]{BWW02}, there exist $v_1,v_2\in\Coz(M)$ such that $v_i\pprec u_i$ for $i=1,2$ and $v_1\vee v_2=1$. Then by \cref{zero-between}, there exist $y_i \in \Z(M)$ such that $v_i \leq y_i \leq u_i$ for $i = 1,2$. It follows that $z_i \leq \neg y_i \in \Coz(M)$ for $i=1,2$ and $\neg y_1, \neg y_2$ are disjoint cozero-elements. Since the cozero-elements are precisely the basic open elements associated with $\Z(M)$, \cref{lem:Wallman-normality} implies that $\beta M$ is Hausdorff.

If $M$ is compact, then $M$ is compact $T_1$, so $M \cong \beta M$ by \cref{prop:WM-unique}.
\end{proof}

\begin{corollary}[Choice-free Stone--\v{C}ech compactification]
    Let $M$ be a spatial completely regular MT-algebra. Then
    $(\beta M,\rho_{\beta})$, where
    \[
        \rho_{\beta}(A)
        =
        \bigvee\{a\in\at(M)\mid
        \sigma_{\Z(M)}(a)\subseteq A\},
    \]
    is a Hausdorff compactification of $M$.
\end{corollary}

\begin{proof}
    By \cref{lem:zero-wallman-basis}, $\Z(M)$ is a Wallman basis of
    $M$. Hence, by \cref{thm:choice-free-wallman},
    $(\beta M,\rho_{\beta})$ is a $T_1$-compactification of $M$.
    By \cref{prop:beta-hausdorff}, $\beta M$ is Hausdorff.
    Therefore, $(\beta M,\rho_{\beta})$ is a Hausdorff
    compactification of $M$.
\end{proof}

We now relate the Stone--\v{C}ech extension of an MT-algebra to the analoguous extension for frames. The \emph{Stone--\v{C}ech extension} of a completely regular frame $L$ is obtained by the frame $\beta L$ of \emph{completely regular ideals} of $L$ (\cite{BM80}; see also \cite[p.~132]{PP21}), where we recall that an ideal $I \in \J(L)$ is \emph{completely regular} if $a \in I$ implies there exists $b \in I$ such that $a \pprec b$. Then $(\beta L, \nu)$ is a regular compactification of $L$, where $\nu : \beta L \to L$ is the join map. To show that $\O(\beta M) \cong \beta \O(M)$, we require the following results.

\begin{lemma}\label{lem:zero-cozero-pprec} 
	Let $M$ be completely regular. 
	If $z\in\Z(M)$, $v\in\Coz(M)$, and $z\leq v$, then $z\pprec v$. 
\end{lemma} 
\begin{proof} 
	Since $\neg z,v\in\Coz(M)$ and $\neg z\vee v=1$, by \cite[Cor.~5.1.3]{BWW02} there exist $u_1,u_2\in\Coz(M)$ such that 
	$u_1\pprec\neg z$, $u_2\pprec v$, and $u_1\vee u_2=1$. 
    Since $u_1\leq\neg z$, we have $z\leq\neg u_1$. Moreover, $u_1\vee u_2=1$ implies $\neg u_1\leq u_2$. Hence, $z\leq\neg u_1\leq u_2\pprec v$, and therefore $z\pprec v$.
\end{proof} 

\begin{lemma}\label{lem:sigma-pprec} 
	Let $M$ be completely regular, $u\in\O(M)$ and $v\in\Coz(M)$. 
	If $\diamond \sigma_{\Z(M)}(u)\leq\sigma_{\Z(M)}(v)$, then $u\pprec v$. 
\end{lemma}
\begin{proof} 
	Since the elements $\sigma(z)$, for $z\in\Z(M)$, meet-generate $\C(\beta M)$, 
	\[ 
		\diamond\sigma(u) = \bigwedge\{\sigma(z)\mid z\in\Z(M) \text{ and }\sigma(u)\leq\sigma(z)\}.
	\] 
	Since $\neg v\in\Z(M)$ and $\diamond \sigma(u)\wedge\sigma(\neg v)=0$, compactness of $\beta M$ yields $z_1,\ldots,z_n\in\Z(M)$ such that $\sigma(u)\leq\sigma(z_i)$ for each $i$ and $\sigma(z_1)\wedge\cdots\wedge\sigma(z_n) \leq\sigma(v)$.
	Put $z=z_1\wedge\cdots\wedge z_n$. Then $z\in\Z(M)$ and, by \cref{sigma-properties}, 
	$\sigma(u)\leq\sigma(z)\leq\sigma(v)$.
	We claim that $u\leq z$. 
	Otherwise, $u\wedge\neg z\neq0$. 
	Since $\Z(M)$ is a Wallman basis, there is a non-zero $e\in\Z(M)$ such that $e\leq u\wedge\neg z$. 
	Then $\<e\rangle\in\sigma(u)$, whereas $\<e\rangle\notin\sigma(z)$, a contradiction. 
	Thus $u\leq z$. 
	Moreover, $z\leq v$ by \cref{lem:injective-basics}. 
	Hence, $u\leq z\pprec v$ by \cref{lem:zero-cozero-pprec}, and therefore $u\pprec v$. 
\end{proof}

\begin{theorem}
Let $M$ be a completely regular MT-algebra. Then $\O(\beta M)\cong \beta\O(M)$.
\end{theorem}

\begin{proof}
    Define
    \[
        \alpha:\O(\beta M)\to\beta \O(M),
        \qquad
        \alpha(U)=\{u\in \O(M)\mid \diamond \sigma(u)\leq U\},
    \]
    and
    \[
        \gamma:\beta \O(M)\to\O(\beta M),
        \qquad
        \gamma(I)=
        \bigvee\{\sigma(v)\mid v\in I\cap\Coz(M)\}.
    \]

    The map $\gamma$ is well defined because $\sigma(v)$ is open in $\beta M$
    for every $v\in\Coz(M)$. We next show that $\alpha$ is well defined. It is immediate that
    $\alpha(U)$ is an ideal. Let $u\in\alpha(U)$. Since $\beta M$ is compact
    Hausdorff, it is regular by \cref{compact-eqv-sep}.
    Thus, there exists $V\in\O(\beta M)$ such that
    $\diamond\sigma(u)\leq V$ and $ \diamond V\leq U$.
    The closed element $\diamond\sigma(u)$ is compact, and the elements
    $\sigma(v)$, with $v\in\Coz(M)$, join-generate $\O(\beta M)$. Hence, there
    exists $v\in\Coz(M)$ such that
    $\diamond\sigma(u)\leq\sigma(v)\leq V$.
    Therefore, $\diamond\sigma(v)\leq U$, so $v\in\alpha(U)$. 
    By \cref{lem:sigma-pprec}, $u\pprec v$. Hence $\alpha(U)$ is a
    completely regular ideal.
    
    We show that $\gamma\alpha(U) = U$. Clearly,
    $\gamma(\alpha(U))\leq U$. Conversely, regularity of $\beta M$ gives
    \[
        U=\bigvee\{V\in\O(\beta M)\mid \diamond V\leq U\}.
    \]
    If $\diamond V\leq U$, then $V$ is a join of basic opens $\sigma(v)\leq V$, where $v\in\Coz(M)$. For each such $v$,
    $\diamond\sigma(v)\leq\diamond V\leq U$,
    so $v\in\alpha(U)$. Hence, $V\leq\gamma(\alpha(U))$, and therefore
    $U\leq\gamma(\alpha(U))$.
    
    We next show that $\alpha\gamma(I) = I$. Let $I\in\beta \O(M)$ and
    $u\in I$. Then there exists $a\in I$ such that $u\pprec a$. Since $\pprec$ interpolates,
    there exists $b$ such that $u\pprec b\pprec a$. Applying
    \cref{zero-between} to $\neg a\pprec\neg b$, there exists $z\in\Z(M)$ such that $\neg a\leq z\leq\neg b$. Then
    $v=\neg z\in\Coz(M)$, $b\leq v\leq a$, and hence
    $u\pprec v\in I$.
    A further application of \cref{zero-between} gives $c\in\Z(M)$ such
    that $u\leq c\leq v$. 
    Therefore
    $\diamond\sigma(u)\leq\sigma(c)\leq\sigma(v)\leq\gamma(I)$,
    so $u\in\alpha(\gamma(I))$. Thus,
    $I\subseteq\alpha(\gamma(I))$.
    Conversely, suppose $u\in\alpha(\gamma(I))$. Since
    $\diamond\sigma(u)$ is compact, there exist
    $v_1,\ldots,v_n\in I\cap\Coz(M)$ such that
    \[
        \diamond\sigma(u)
        \leq
        \sigma(v_1)\vee\cdots\vee\sigma(v_n)
        =
        \sigma(v_1\vee\cdots\vee v_n).
    \]
    Put $v=v_1\vee\cdots\vee v_n$. Then $v\in I\cap\Coz(M)$, and
    \cref{lem:sigma-pprec} gives $u\pprec v$. In particular $u\leq v$,
    so $u\in I$. Hence, $\alpha(\gamma(I))\subseteq I$.
    
    Thus $\alpha$ and $\gamma$ are mutually inverse order isomorphisms, and therefore frame isomorphisms. Consequently, $\O(\beta M)\cong\beta\O(M)$.
\end{proof}

We obtain the following corollary.

\begin{corollary}
    Let $M$ be a completely regular MT-algebra. Then $(\O(\beta M), \rho)$ is the Stone--\v{C}ech compactification of $\O(M)$, where $\rho : \O(\beta M) \to \O(M)$ is the map \[
        \rho(U) = \bigvee\{u\in \O(M) \mid \diamond \sigma(u) \subseteq U\}.
    \]
\end{corollary}

We now show that under choice the Stone--\v{C}ech extension of a spatial MT-algebra coincides with the classic Stone--{\v C}ech extension of completely regular space. For this we recall (see, e.g., \cite[Ch.~6]{GJ60}) the following convenient description of the latter. Let $X$ be a completely regular space. Then the space $\beta X = \max(\Z(\P(X)))$ is the \emph{Stone--{\v C}ech extension} of $X$ and the pair $(\beta X, e)$, where $e : X \to \beta X$ is defined by $e(x) = \{Z \in \Z(\P(X)) \mid x \in Z\}$ is the \emph{Stone--\v{C}ech compactification} of $X$.
\begin{*theorem}
    Let $X$ be a completely regular space.  Then $\P(\beta X) \cong \beta \P(X)$. Consequently, $\beta X \cong \at(\beta \P(X))$.
\end{*theorem}

\begin{proof}
By definition, $\beta\P(X)=\RO(X_{\Z(\P(X))})$. Hence, by \cref{prop:RO-max},
$\beta\P(X)\cong\P(\max X_{\Z(\P(X))})=\P(\beta X)$ as MT-algebras. Applying $\at$ yields $\at(\beta\P(X))\cong\beta X$.
\end{proof}

We finish the paper with an extension of the celebrated result that the Wallman extension and Stone--\v{C}ech extension in the normal case coincide. We require the following lemmas.

\begin{lemma}\label{wallman-basis-comparison}
Let $\B\subseteq \mathcal D$ be Wallman bases of $M$. Suppose that for all
$d,e\in \mathcal D$ with $d\wedge e=0$, there is $b\in \B$ such that
$d\le b$ and $b\wedge e=0$. Then $x \in \sigma_{\mathcal D}(d)$ iff
$x \cap \B \in \sigma_{\B}(d)$, for all $x \in X_{\mathcal D}$
and $d \in \mathcal D$.
\end{lemma}

\begin{proof}
Let $x' = x \cap \B$. Since $\B \subseteq \mathcal D$ is a bounded
sublattice and $x$ is a proper filter of $\mathcal D$, we have $x'\in X_{\B}$.

($\Rightarrow$) Suppose $x\in\sigma_{\mathcal D}(d)$, and let $b\in\B$ with
$b\wedge d=0$. Since $b\in\mathcal D$, \cref{lem:sigma-def} gives
$e\in x$ with $e\wedge b=0$. By the assumption, there is $c\in\B$ such that
$e\le c$ and $c\wedge b=0$. Since $x$ is upward closed, $c\in x$, and hence
$c\in x'$. Thus $x'\in\sigma_{\B}(d)$.

($\Leftarrow$) Suppose $x'\in\sigma_{\B}(d)$, and let $e\in\mathcal D$ with
$e\wedge d=0$. By the assumption, there is $b\in\B$
such that $e\le b$ and $b\wedge d=0$. Since $x'\in\sigma_{\B}(d)$,
\cref{lem:sigma-def} gives $c\in x'$ with $c\wedge b=0$. Then $c\in x$ and
$c\wedge e=0$, since $e\le b$. Therefore $x\in\sigma_{\mathcal D}(d)$.
\end{proof}

\begin{lemma} \label{ZM-and-CM}
Let $M$ be normal and completely regular, $c \in \C(M)$, and
$x \in X_{\C(M)}$. Then $x \in \sigma_{\C(M)}(c)$ iff
$x \cap \Z M \in \sigma_{\Z M}(c)$.
\end{lemma}

\begin{proof}
Since $M$ is completely regular, $\Z M$ is a Wallman basis. We show that the
hypothesis of \cref{wallman-basis-comparison} is satisfied for
$\B=\Z M$ and $\mathcal D=\C(M)$. Let $d,e\in\C(M)$ with $d\wedge e=0$.
Then $d\leq \neg e$. By normality and complete regularity, there is
$z\in\Z M$ such that $d\leq z\leq \neg e$. Thus $z\wedge e=0$.
The result now follows from \cref{wallman-basis-comparison}.
\end{proof}

\begin{theorem}
    Let $M$ be a completely regular normal MT-algebra. Then $\beta M \cong \w M$. 
\end{theorem}

\begin{proof}
Since both $\beta M$ and $\w M$ are $T_1$, it suffices to show that $\C(\w M)\cong\C(\beta M)$ by \cref{lem:T1-iso}. Define $\varphi:\C(\w M)\to\C(\beta M)$ by
\[
     \varphi(C)=\bigcap\{\sigma_{\Z M}(f)\mid f\in\Z M\text{ and }C\subseteq\sigma_{\C(M)}(f)\}. 
\]
It is clear that $\varphi$ is order-preserving. To see that $\varphi$ reflects order, suppose $D\nsubseteq C$ for $C,D\in\C(\w M)$. Then there exists $c\in\C(M)$ such that $C\subseteq\sigma_{\C(M)}(c)$ but $D\nsubseteq\sigma_{\C(M)}(c)$. Thus, there exists $x\in D$ such that $x\notin\sigma_{\C(M)}(c)$. By \cref{ZM-and-CM}, $x':=x\cap\Z M\notin\sigma_{\Z M}(c)$. By the proof of \cref{ZM-and-CM}, there exists $f\in\Z M$ such that $c\leq f$ and $x'\notin\sigma_{\Z M}(f)$. Since $C\subseteq\sigma_{\C(M)}(c)\subseteq\sigma_{\C(M)}(f)$, we have $x'\notin\varphi(C)$. On the other hand, if $g\in\Z M$ and $D\subseteq\sigma_{\C(M)}(g)$, then $x\in\sigma_{\C(M)}(g)$, so \cref{ZM-and-CM} gives $x'\in\sigma_{\Z M}(g)$. Hence, $x'\in\varphi(D)$. Therefore, $\varphi(D)\nsubseteq\varphi(C)$, and $\varphi$ reflects order.

To see that $\varphi$ is surjective, let $C\in\C(\beta M)$ and set 
\[
    C'=\bigcap\{\sigma_{\C(M)}(f)\mid f\in\Z M\text{ and }C\subseteq\sigma_{\Z M}(f)\}. 
\]
Then $C'\in\C(\w M)$ because $\Z M\subseteq\C(M)$. We show that $\varphi(C')=C$. If $C\subseteq\sigma_{\Z M}(f)$, then $C'\subseteq\sigma_{\C(M)}(f)$, so $\varphi(C')\subseteq C$, since the closed elements of $\beta M$ are meet-generated by the $\sigma_{\Z M}(f)$. Conversely, let $x'\in C$ and put $x=\{c\in\C(M)\mid z\leq c\text{ for some }z\in x'\}$. Then $x\in X_{\C(M)}$ and $x\cap\Z M=x'$. If $f\in\Z M$ and $C\subseteq\sigma_{\Z M}(f)$, then $x'\in\sigma_{\Z M}(f)$, so \cref{ZM-and-CM} gives $x\in\sigma_{\C(M)}(f)$. Thus, $x\in C'$. If $f\in\Z M$ and $C'\subseteq\sigma_{\C(M)}(f)$, then $x\in\sigma_{\C(M)}(f)$, and another application of \cref{ZM-and-CM} gives $x'\in\sigma_{\Z M}(f)$. Therefore, $x'\in\varphi(C')$, so $C\subseteq\varphi(C')$. Hence, $\varphi(C')=C$. Thus, $\varphi$ is an order isomorphism, and therefore $\w M\cong\beta M$ by \cref{lem:T1-iso}.
\end{proof}

The Wallman basis construction therefore yields choice-free compactifications
for spatial MT-algebras: every spatial $T_1$ MT-algebra admits a Wallman
compactification, and every spatial completely regular MT-algebra admits a Stone--\v{C}ech compactification. In particular, every $T_1$-space
and every completely regular $T_1$-space, represented by its powerset
MT-algebra, admits the corresponding compactification in the category of
MT-algebras without invoking choice. Assuming choice, these compactifying
algebras are spatial and recover the usual compactifying spaces. For
non-spatial $T_1$-algebras, however, the Wallman extension cannot in general
be a compactification in the sense of \cref{def:compactification}. This leaves
open whether there is a natural weaker notion under which the Wallman-type
extensions of arbitrary $T_1$ MT-algebras become compactifications.

\end{document}